\documentclass[a4paper,10pt]{article}
\usepackage[ansinew]{inputenc}          
\usepackage{amssymb,amsthm,amsmath}     
\usepackage[dvips]{graphicx}            
\usepackage[all]{xy}
\usepackage{subfigure}
\usepackage{algorithm,algorithmic}

\theoremstyle{plain}
\newtheorem{thm}{Theorem}[section]

\newtheorem{lemma}[thm]{Lemma}

\theoremstyle{definition}
\newtheorem{definition}[thm]{Definition}
\newtheorem{example}[thm]{Example}

\newcommand{\Z}{\mathbb{Z}}

\newcommand{\bu}{\mathbf{u}}
\newcommand{\bv}{\mathbf{v}}
\newcommand{\bc}{\mathbf{c}}

\newcommand{\ms}{\bar{s}}

\begin{document}

\title{Optimal lower bound for $2$-identifying code in the hexagonal grid\thanks{This paper has been presented in part at the International Workshop on Coding and Cryptography 2011, WCC 2011.}}

\author{\textbf{Ville Junnila}\thanks{Research supported by the Academy of Finland under grant 210280.} \\
Turku Centre for Computer Science TUCS and\\
Department of Mathematics\\
University of Turku, FI-20014 Turku, Finland\\
viljun@utu.fi \and \textbf{Tero Laihonen}\\
Department of Mathematics\\
University of Turku, FI-20014 Turku, Finland\\
terolai@utu.fi}
\date{}
\maketitle

\begin{abstract}
An $r$-identifying code in a graph $G = (V,E)$ is a subset $C \subseteq V$ such that for each $u \in V$ the intersection of $C$
and the ball of radius $r$ centered at $u$ is non-empty and unique.
Previously, $r$-identifying codes have been studied in various grids. In particular, it has been shown that there exists a $2$-identifying
code in the hexagonal grid with density $4/19$ and that there are no
$2$-identifying codes with density smaller than $2/11$. Recently, the lower bound has been improved to $1/5$ by Martin and Stanton (2010).
In this paper, we prove that the $2$-identifying code with density
$4/19$ is optimal, i.e. that there does not exist a $2$-identifying
code in the hexagonal grid with smaller density.
\end{abstract}

\section{Introduction}

Let $G = (V, E)$ be a simple connected and undirected graph with $V$
as the set of vertices and $E$ as the set of edges. Let $u$ and $v$
be vertices in $V$. If $u$ and $v$ are adjacent to each other, then
the edge between $u$ and $v$ is denoted by $\{u, v\}$ (or in short by $uv$). The \emph{distance} between $u$ and $v$ is denoted by $d(u,v)$ and is defined as the number of edges in any shortest path between
$u$ and $v$. Let $r$ be a positive integer. We say that $u$
$r$\emph{-covers} $v$ if the distance $d(u,v)$ is at most $r$. The
\emph{ball of radius} $r$ \emph{centered at} $u$ is defined as
\[
B_r(u) = \{ x \in V \ | \ d(u,x) \leq r \} \textrm{.}
\]

A non-empty subset of $V$ is called a \emph{code} in $G$, and its
elements are called \emph{codewords}. Let $C \subseteq V$ be a code
and $u$ be a vertex in $V$. An \emph{I-set} (or an \emph{identifying
set}) of the vertex $u$ with respect to the code $C$ is defined as
\[
I_r(C;u) = I_r(u) = B_r(u) \cap C \textrm{.}
\]
The following definition is due to Karpovsky \emph{et al.}
\cite{kcl}.
\begin{definition}
Let $r$ be a positive integer. A code $C \subseteq V$ is said to be
$r$-\emph{identifying} in $G$ if for all $u,v \in V$ ($u \neq v$) the set $I_r(C;u)$ is non-empty and
\[
I_r(C;u) \neq I_r(C;v) \textrm{.}
\]
\end{definition}

Let $X$ and $Y$ be subsets of $V$. The \emph{symmetric difference} of $X$ and $Y$ is defined as $X \, \triangle \, Y = (X \setminus Y) \cup (Y \setminus X)$. We say that the vertices $u$ and $v$ are $r$\emph{-separated} by a code $C \subseteq V$ (or by a codeword of
$C$) if the symmetric difference $I_r(C;u) \, \triangle \, I_r(C;v)$
is non-empty. The definition of $r$-identifying codes can now be
reformulated as follows: $C \subseteq V$ is an $r$-identifying code
in $G$ if and only if for all $u,v \in V$ ($u \neq v$) the vertex $u$ is $r$-covered by a codeword of $C$ and 
\[
I_r(C;u) \, \triangle \, I_r(C;v) \neq \emptyset \textrm{.}
\]

In this paper, we study identifying codes in the hexagonal grid. We
define the hexagonal grid $G_H = (V_H, E_H)$ using the brick wall
representation as follows: the set of vertices $V_H = \Z^2$ and the
set of edges
\[
E_H = \{ \{ \bu=(i,j), \bv \} \ | \ \bu, \bv \in \Z^2, \bu - \bv \in \{(0,(-1)^{i+j+1}), (\pm 1, 0)\} \} \textrm{.}
\]
This definition is illustrated in Figure~\ref{HexBrick}. The
hexagonal grid can also be illustrated using the honeycomb
representation as in Figure~\ref{HexHoneycomb}. In both
illustrations, lines represent the edges and intersections of the
lines represent the vertices of $G_H$. The labeling of the
vertices in the brick wall representation is self-explanatory. This
labeling can also be applied to the honeycomb representation, if
we visualize the honeycomb to be obtained from the brick wall by
squeezing it from left and right. For an example of the labeling
of the vertices, we refer to Figure~\ref{HexDefinitionIllustrated}.

\begin{figure}[ht]
\centering
\subfigure[Brick wall]{
\includegraphics[height=95pt]{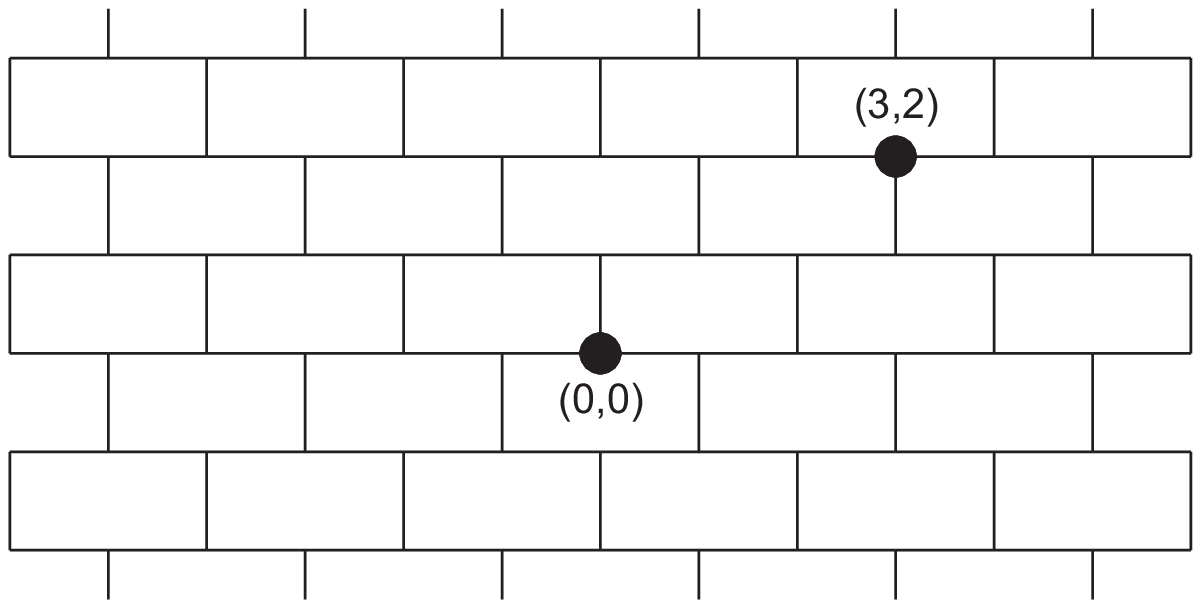}
\label{HexBrick}
}
\subfigure[Honeycomb]{
\includegraphics[height=95pt]{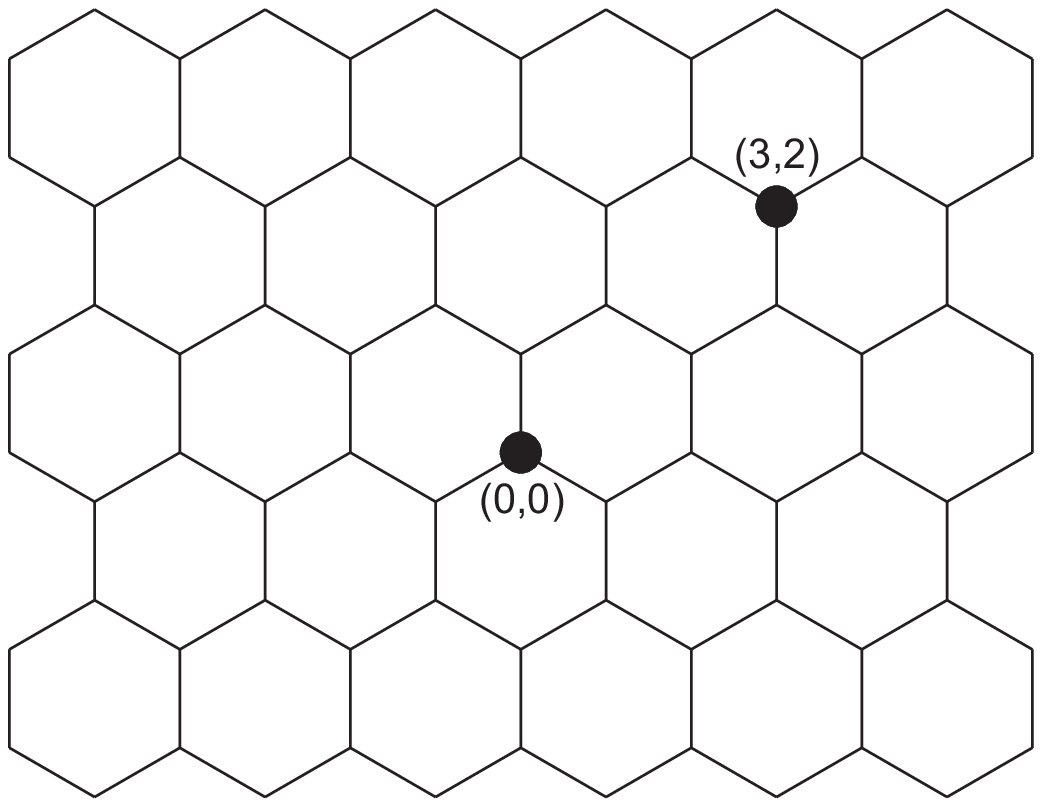}
\label{HexHoneycomb}
}
\caption{The brick wall and the honeycomb representations illustrated.} 
\label{HexDefinitionIllustrated}
\end{figure}

To measure the size of an identifying code in the infinite hexagonal
grid, we introduce the notion of density. For the formal definition,
we first define
\[
Q_n = \{ (x,y) \in V_H \ | \ |x| \leq n, |y| \leq n \} \textrm{.}
\]
Then the \emph{density} of a code $C \subseteq V_H$ is defined as
\[
D(C) = \limsup_{n \rightarrow \infty} \frac{|C \cap Q_n|}{|Q_n|} \textrm{.}
\]
Naturally, we try to construct identifying codes with as small
density as possible. Moreover, we say that an $r$-identifying code is \emph{optimal}, if there do not exist any $r$-identifying codes with
smaller density.

Previously, $r$-identifying codes in $G_H$ have been studied in
various papers. The first results concerning $r$-identifying codes in $G_H$ have been presented in the seminal paper \cite{kcl} in the case $r=1$. Later these results have been improved by showing that there
exists a $1$-identifying code with density $3/7$ (see Cohen \emph{et al.} \cite{chlz3}) and that there do not exist $1$-identifying codes in $G_H$ with density smaller than $12/29$ (see Cranston and Yu \cite{CYlbh}). For general $r \geq 2$, Charon \emph{et al.} \cite{chhl} showed that each $r$-identifying code $C$ in $G_H$ has $D(C) \geq 2/(5r+3)$ if $r$ is even and $D(C) \geq 2/(5r+2)$ if $r$ is odd. They also presented a construction for each $r \geq 2$ giving an $r$-identifying code $C \subseteq V_H$ with $D(C) \sim 8/(9r)$.

For small values of $r$, the previous constructions have been
improved in \cite{chl} by Charon \emph{et al}. In particular, it is shown that there exists a $2$-identifying code in $G_H$ with density $4/19$. 
In the case $r=2$, the aforementioned general lower bound is improved in Martin and Stanton \cite{MSlbg} by showing that the density of any $2$-identifying code in $G_H$ is at least $1/5$. In this paper, we further improve this lower bound to $4/19$. In other words, we show that the previously presented $2$-identifying code with density $4/19$ is optimal. (In \cite{MSlbg}, it is also shown that there do not exist any $2$-identifying codes in the \emph{square grid} with density smaller than $6/37$. In a forthcoming paper, we are going to improve this lower bound using somewhat different ideas as in this paper.)

The organization of the paper is as follows. In
Section~\ref{HexBasics}, we first introduce some preliminary
definitions and results. Then, in Section~\ref{HexLowerBound}, we
proceed with the actual proof of the lower bound.


\section{Basics} \label{HexBasics}

Let $G = (V, E)$ be a simple connected and undirected graph. Assume
also that $C$ is a code in $G$. The following concept of the share of a codeword has been introduced by Slater in \cite{S:fault-tolerant}.
The \emph{share} of a codeword $c \in C$ is denoted by $s_r(c)$ and
defined as
\[
s_r(C; c) = s_r(c) = \sum_{u \in B_r(c)} \frac{1}{|I_r(C; u)|} \textrm{.}
\]
The notion of share proves to be useful in determining lower bounds
of $r$-identifying codes (as explained in the following paragraph).

Assume that $G=(V,E)$ is a finite graph and $D$ is a code in $G$ such that $B_r(u) \cap D$ is non-empty for all $u \in V$. Then it is easy to conclude that $\sum_{c \in D} s_r(D;c) = |V|$. Assume further that $s_r(D;c) \leq \alpha$ for all $c \in D$. Then we have $|V| \leq
\alpha |D|$, which immediately implies
\[
|D| \geq \frac{1}{\alpha} |V| \textrm{.}
\]
Assume then that for any $r$-identifying code $C$ in $G$ we have
$s_r(C;c) \leq \alpha$ for all $c \in C$. By the aforementioned
observation, we then obtain the lower bound $|V|/\alpha$ for the size of an $r$-identifying code in $G$. In other words, by determining the maximum share for any $r$-identifying code, we obtain a lower bound
for the minimum size of an $r$-identifying code.

The previous reasoning can also be generalized to the case when an
infinite graph is considered. In particular, if for any
$r$-identifying code in $G_H$ we have $s_r(C;c) \leq \alpha$ for all
$c \in C$, then it can be shown that the density of an
$r$-identifying code in $G_H$ is at least $1/\alpha$ (compare to
Theorem~\ref{HexR2MainThm}). The main idea behind the proof of the
lower bound (in Section~\ref{HexLowerBound}) is based on this
observation, although we use a more sophisticated method by showing
that for any $2$-identifying code the share is on \emph{average} at
most $19/4$. In Theorem~\ref{HexR2MainThm}, we present a formal proof to verify that this method is indeed valid.

\medskip

In the proof of the lower bound, we need to determine upper bounds for shares of codewords. 
To formally present a way to estimate shares, we first need to introduce some notations.

Let $D \subseteq V$ be a code and $c$ be a codeword of $D$. Consider then the $I$-sets $I_r(D;u)$ when $u$ goes through all the vertices in $B_r(c)$. (Notice that all of these $I$-sets do not have to be different.) Denote the different identifying sets by $I_1, I_2, \ldots, I_k$, where $k$ is a positive integer. Furthermore, denote the number of identifying sets equal to $I_j$ by $i_j$ $(j=1,2,\ldots,k)$. Now we are ready to present the following lemma, which provides a method to estimate the shares of the codewords.
\begin{lemma} \label{SimpleEstLemma}
Let $C$ be an $r$-identifying code in $G$ and let $D$ be a non-empty subset of $C$. For $c \in D$, using the previous notations, we have
\[
s_r(C;c) \leq \sum_{j=1}^{k} \left( \frac{1}{|I_j|} + (i_j-1)\frac{1}{|I_j|+1} \right) \textrm{.}
\]
\end{lemma}
\begin{proof}
Assume that $c \in D$. Then, for each $j=1,2,\ldots,k$, define $\mathcal{I}_j = \{ u \in B_r(c) \, | \, I_j=I_r(D;u) \}$. 
Now it is obvious that for at most one vertex $u \in \mathcal{I}_j$ we have $I_j = I_r(C;u)$ and the other vertices of $\mathcal{I}_j$ are $r$-covered by at least $|I_j|+1$ codewords of $C$. Hence, the claim immediately follows.
\end{proof}


The previous lemma will be used numerous times in the following presentation. The computations needed in applying this lemma may sometimes be a little bit tedious, but always very straightforward. It is also quite easy to implement an algorithm to compute the upper bound given by the lemma. Furthermore, the use of the previous lemma is illustrated in the following example.
\begin{figure}
\centering
\subfigure[The first case]{
\includegraphics[height=120pt]{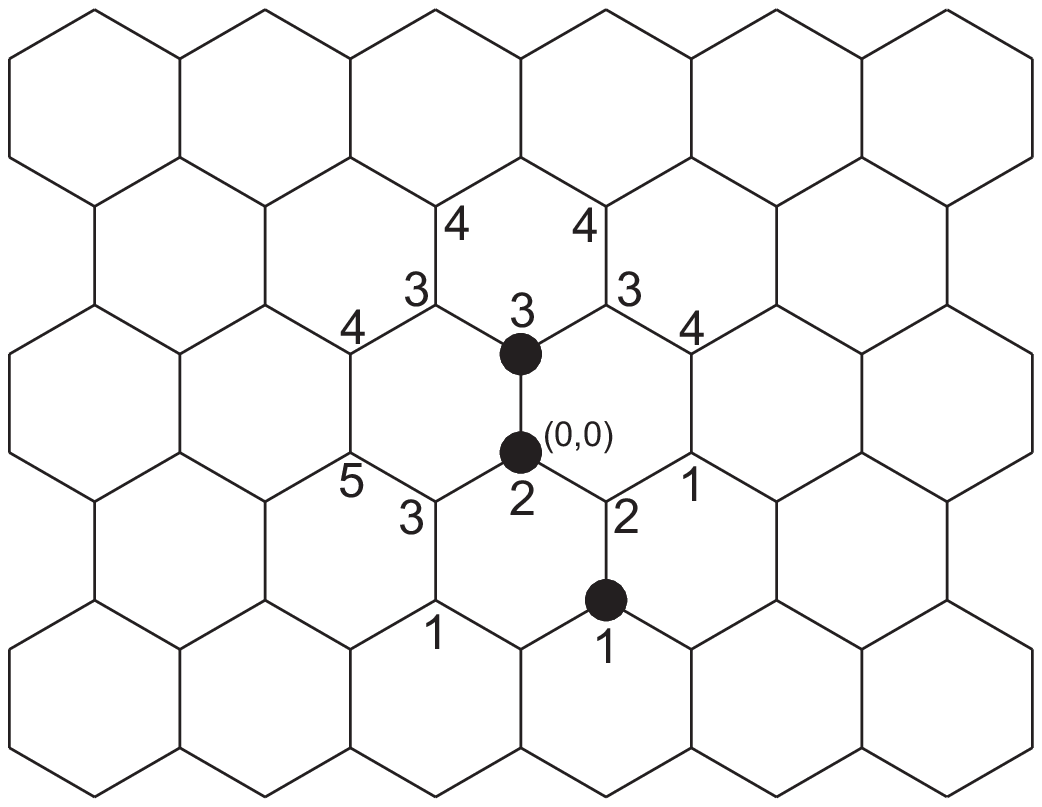}
\label{AdjacentCaseA}
}
\subfigure[The second case]{
\includegraphics[height=120pt]{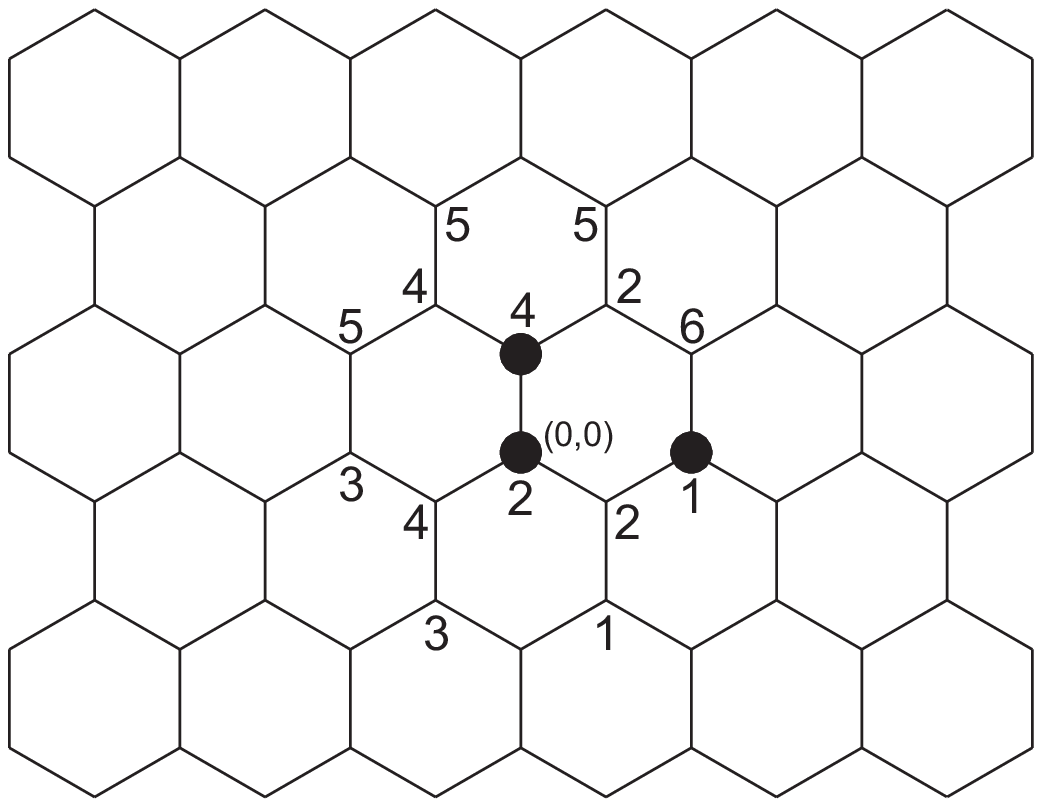}
\label{AdjacentCaseB}
}
\caption{The cases of Example \ref{IllExampleAdjacentLemma} illustrated. The black dots represent codewords of $C$.} 
\label{ExampleIllustrated}
\end{figure}

\begin{example} \label{IllExampleAdjacentLemma}
Let $C$ be a $2$-identifying code in the hexagonal grid $G_H$. For the first case (see Figure~\ref{AdjacentCaseA}), we assume that $D=\{ (0,0), (0,1), (1,-1) \}$ is a subset of $C$. Now we have the following facts:
\begin{itemize}
\item $I_2(D;\bu) = \{(0,0),(1,-1)\}$ for $\bu = (-1,-1), (1,-1)$ or $(2,0)$ (the vertices labeled with $1$ in the figure),
\item $I_2(D;\bu) = \{(0,1),(0,0),(1,-1)\}$ for $\bu = (0,0)$ or $(1,0)$ (the vertices labeled with $2$ in the figure),
\item $I_2(D;\bu) = \{(0,0),(0,1)\}$ for $\bu = (-1,1), (0,1), (1,1)$ or $(-1,0)$ (the vertices labeled with $3$ in the figure),
\item $I_2(D;\bu) = \{(0,1)\}$ for $\bu = (-2,1), (-1,2), (1,2)$ or $(2,1)$ (the vertices labeled with $4$ in the figure), and
\item $I_2(D; (-2,0)) = \{ (0,0) \}$ (the vertex labeled with 5 in the figure).
\end{itemize}
Thus, by Lemma~\ref{SimpleEstLemma}, we obtain that
\[
s_2(C;(0,0)) \leq \left( \frac{1}{2} + 2 \cdot \frac{1}{3} \right) + \left( \frac{1}{3} + \frac{1}{4} \right) + \left( \frac{1}{2} + 3 \cdot \frac{1}{3} \right) + 1 = \frac{17}{4} \textrm{.}
\]
Similarly, we also have
\[
s_2(C;(0,1)) \leq 1 + 4 \cdot \frac{1}{2}+ 4 \cdot \frac{1}{3} + \frac{1}{4} = \frac{55}{12} \textrm{.}
\]

For the second case (see Figure~\ref{AdjacentCaseB}), we assume that $D=\{ (0,0), (0,1), (2,0) \}$ is a subset of $C$. As in the previous case (with the aid of the figure), it can be concluded that
\[
s_2(C;(0,0)) \leq \frac{13}{3} \textrm{ and } s_2(C;(0,1)) \leq \frac{9}{2} \textrm{.}
\]
It should be noted that the results of this example will be later used in the proof of Lemma~\ref{HexR2ReceiveLemma1}.
\end{example}

\section{The proof of the lower bound} \label{HexLowerBound}

For the rest of the section, assume that $C$ is a $2$-identifying code in $G_H$. It can be shown that $s_2(\bc) \leq 5$ for all $\bc \in C$ (see the proof of Lemma~\ref{HexR2ReceivingLemma}). This provides another approach to obtain the lower bound $D(C) \geq 1/5$, which was previously shown in \cite{MSlbg}. In order to improve this lower bound, we need to consider the shares of codewords on average. Indeed, we can show that on average the share of a codeword is at most $19/4$. Therefore, as shown in Theorem~\ref{HexR2MainThm}, we prove that the density $D(C) \geq 4/19$.

The averaging process is done by introducing a \emph{shifting scheme} designed to even out the shares among the codewords of $C$. (Notice that the shifting scheme can also be understood as a discharging method.) The rules of the shifting scheme are defined in Section~\ref{SubsectionRules}. In Section~\ref{SubsectionMainTheorem}, we introduce three lemmas, which state the following results:
\begin{itemize}
\item If $s_2(\bc) > 19/4$ for some $\bc \in C$, then at least $s_2(\bc) - 19/4$ units of share is shifted from $\bc$ to other codewords. (Lemma~\ref{HexR2ReceivingLemma})
\item If share is shifted to a codeword $\bc \in C$, then $s_2(\bc) \leq 19/4$ and the codeword $\bc$ receives at most $19/4 - s_2(\bc)$ units of share. (Lemmas~\ref{HexR2ReceiveLemma2} and \ref{HexR2ReceiveLemma1})
\end{itemize}
In other words, after the shifting is done, the share of each codeword is at most $19/4$. Using this fact, we are able to prove the main theorem (Theorem~\ref{HexR2MainThm}) of the paper according to which $D(C) \geq 4/19$. Finally, in Section~\ref{SubsectionLemmasProofs}, we provide the proofs of the lemmas.

\subsection{The rules of the shifting scheme} \label{SubsectionRules}

The \emph{rules} of the shifting scheme are illustrated in Figure~\ref{HexR2RuleFigure}. Translations, rotations and reflections (over the line passing vertically through $\bu$) can be applied to each rule in such a way that the structure of the underlying graph $G_H$ is preserved. In the rules, share is shifted as follows:
\begin{itemize}
\item In the rules 1, 2, 4 and 7, we shift $1/4$ units of share
    from $\bu$ to $\bv$.
\item In the rule 3,  we shift $1/6$ and $1/12$ units of share
    from $\bu$ to $\bv$ and $\bv'$, respectively.
\item In the rule 5, we shift $1/6$ units of share from $\bu$ to
    $\bv$.
\item In the rules 6, 8, 9 and 10, we shift $1/12$ units of share
    from $\bu$ to $\bv$.
\end{itemize}
We also have the following modifications to the previous rules:
\begin{itemize}
\item If in the rules 1, 2 and 7 we have $\bu + (0,-1) \in C$, then we only shift $1/12$ units of share from $\bu$ to $\bv$ and denote these new rules (respectively) by 1.1, 2.1 and 7.1. Moreover, in the rule 1.1, we shift $1/12$ units of share to $\bv$ whether $(-3,2)$ belongs to $C$ or not.
\item If in the rule 1 we have $\bu + (-3,2) \in C$, then then we shift $1/4$ units of share from $\bu$ to $\bu + (-1,2)$ (no share is shifted to $\bv$) and denote this new rule by 1.2.
\item If in the rule 2 we have $\bu+(-3,1) \in C$, then we shift $1/6$ units of share from $\bu$ to $\bv$ and denote this new rule by 2.2.
\item If in the rule 2 we have $\bu + (1,2) \in C$, then we shift $1/12$ units of share from $\bu$ to $\bv$ and denote this new rule by 2.3.
\end{itemize}

\begin{figure}[htp]
\centering
\subfigure[Rule~1]{
\includegraphics[height=80pt]{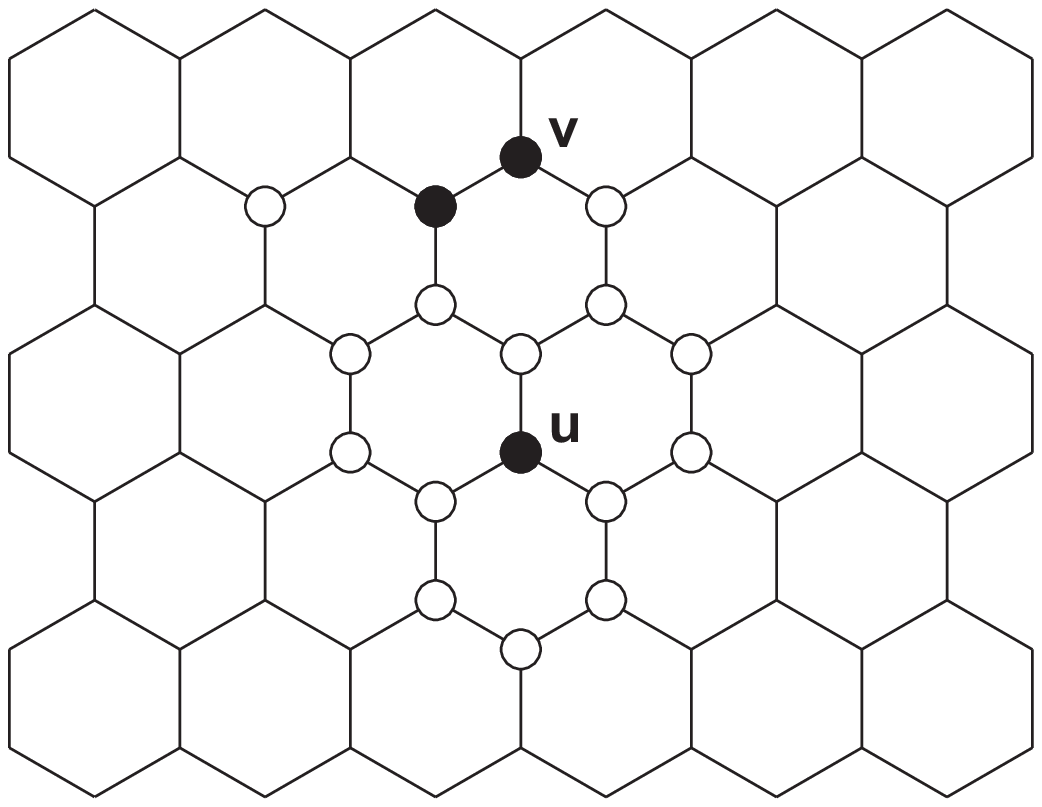}
\label{HexR2Rule1}
}
\subfigure[Rule~2]{
\includegraphics[height=80pt]{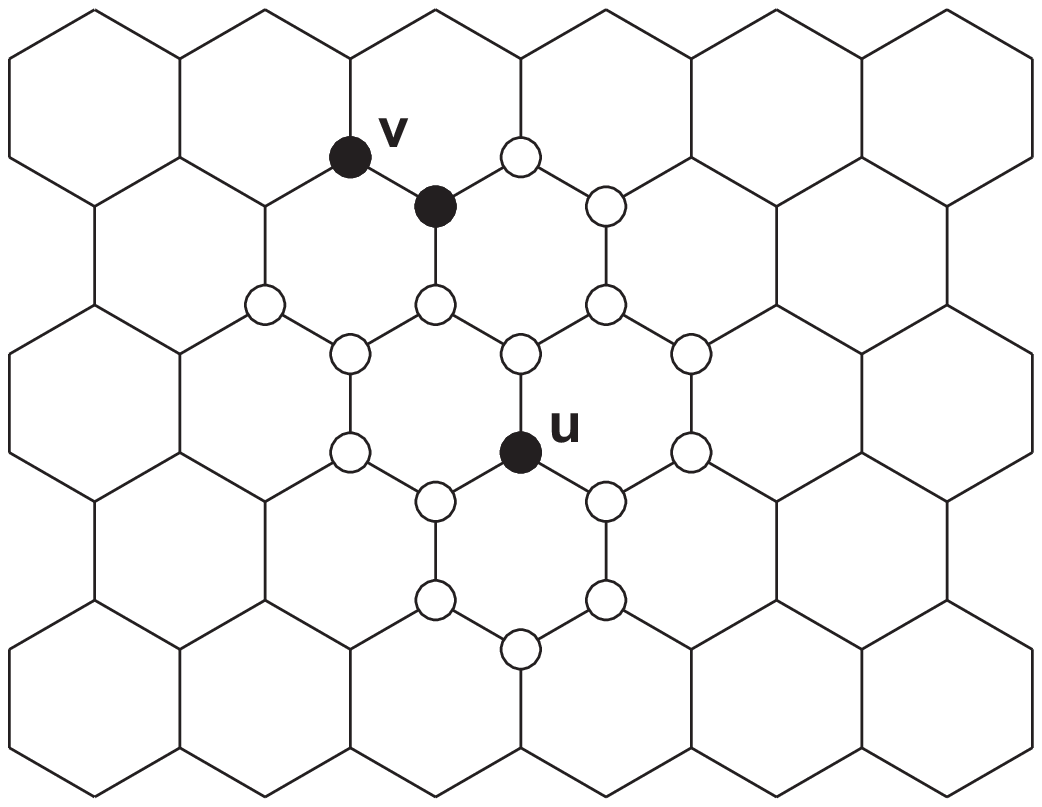}
\label{HexR2Rule2}
}
\subfigure[Rule~3]{
\includegraphics[height=80pt]{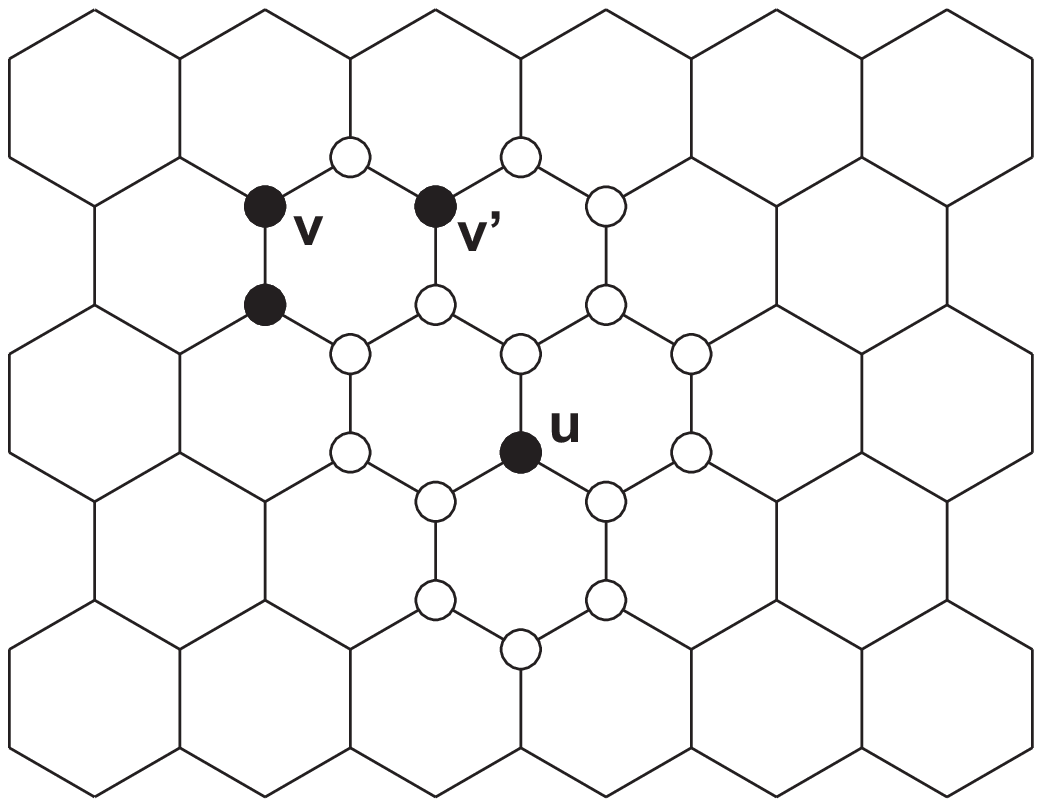}
\label{HexR2Rule3}
}
\subfigure[Rule~4]{
\includegraphics[height=80pt]{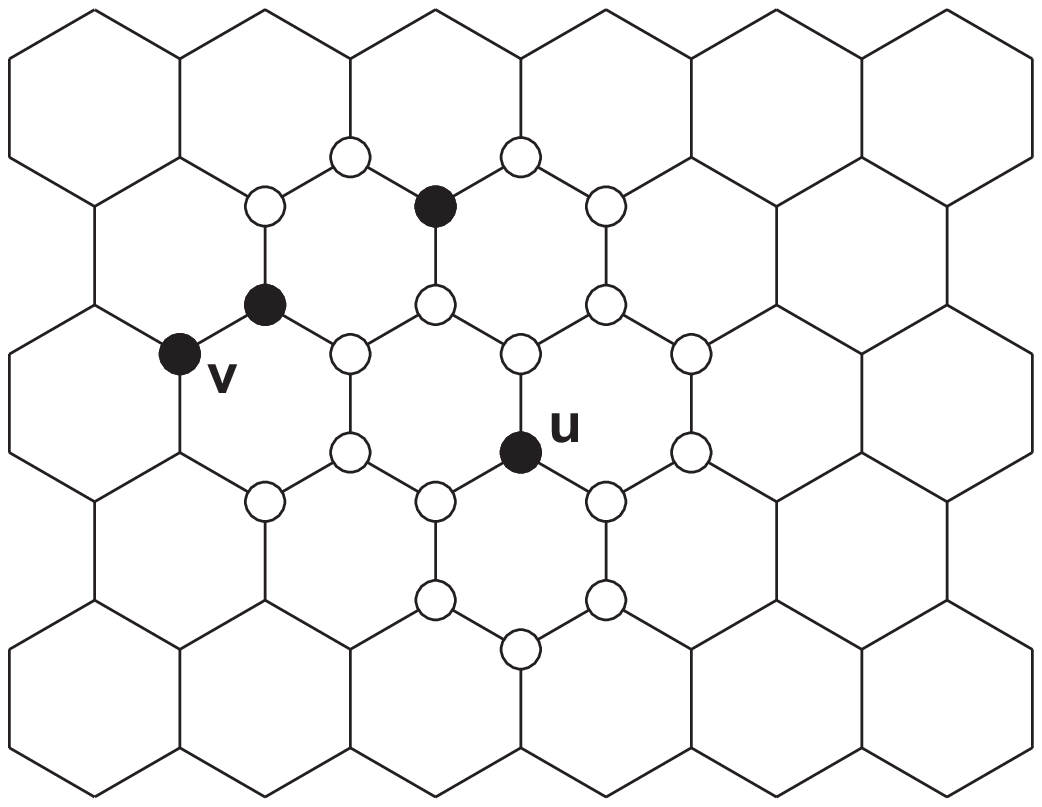}
\label{HexR2Rule4}
}
\subfigure[Rule~5]{
\includegraphics[height=80pt]{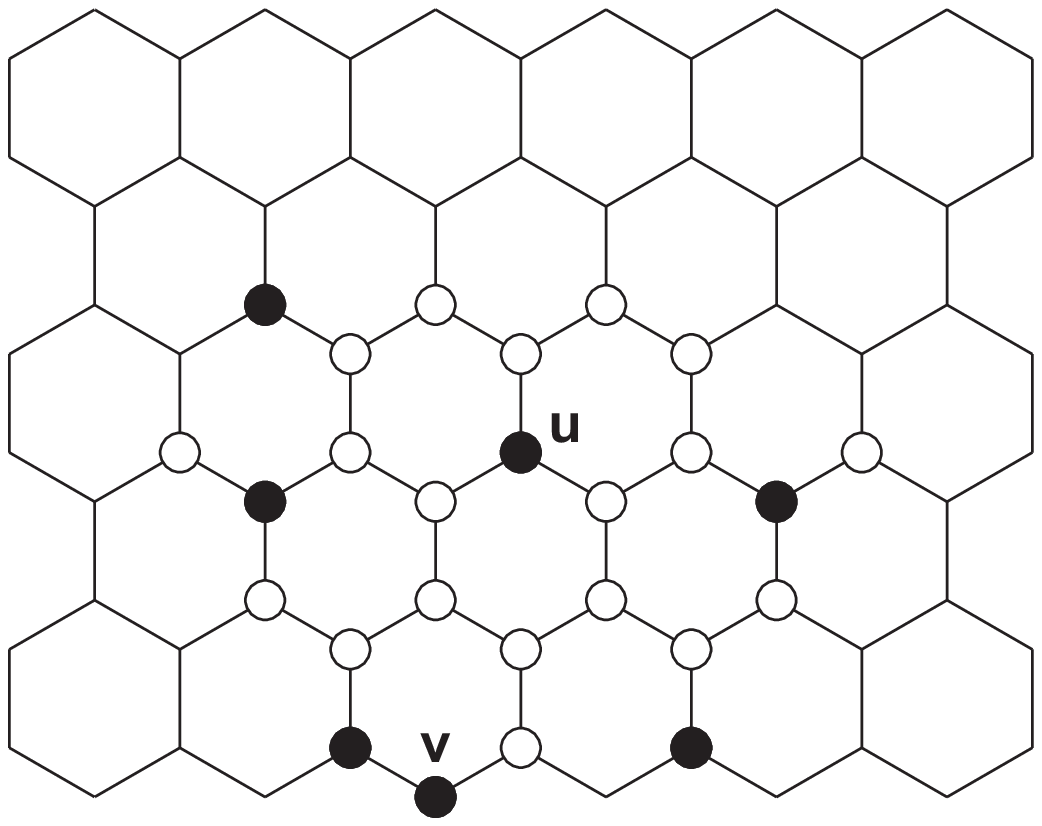}
\label{HexR2Rule5}
}
\subfigure[Rule~6]{
\includegraphics[height=80pt]{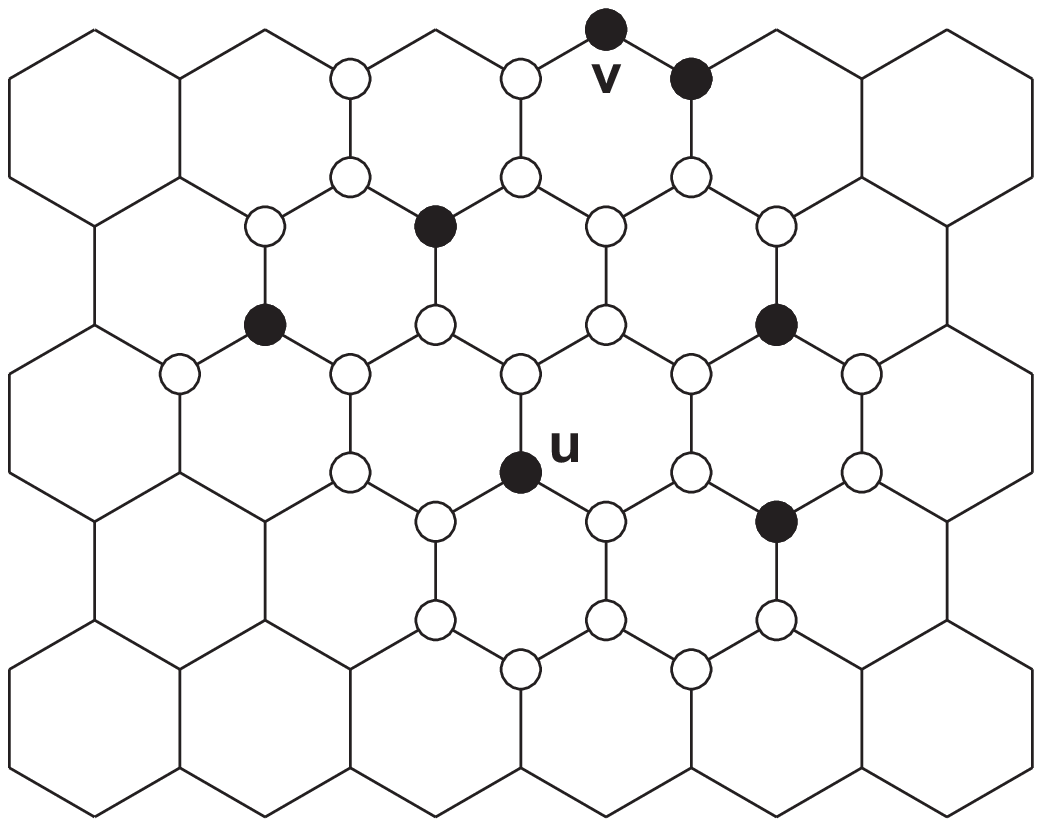}
\label{HexR2Rule6}
}
\subfigure[Rule~7]{
\includegraphics[height=80pt]{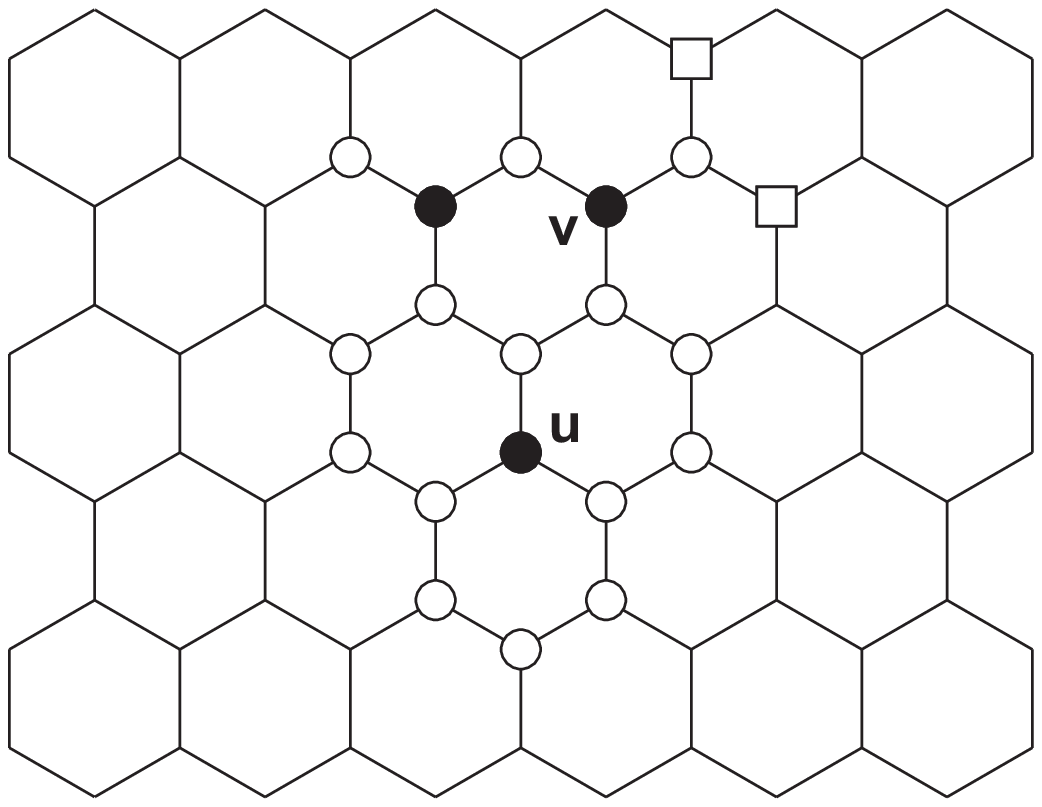}
\label{HexR2Rule7}
}
\subfigure[Rule~8]{
\includegraphics[height=80pt]{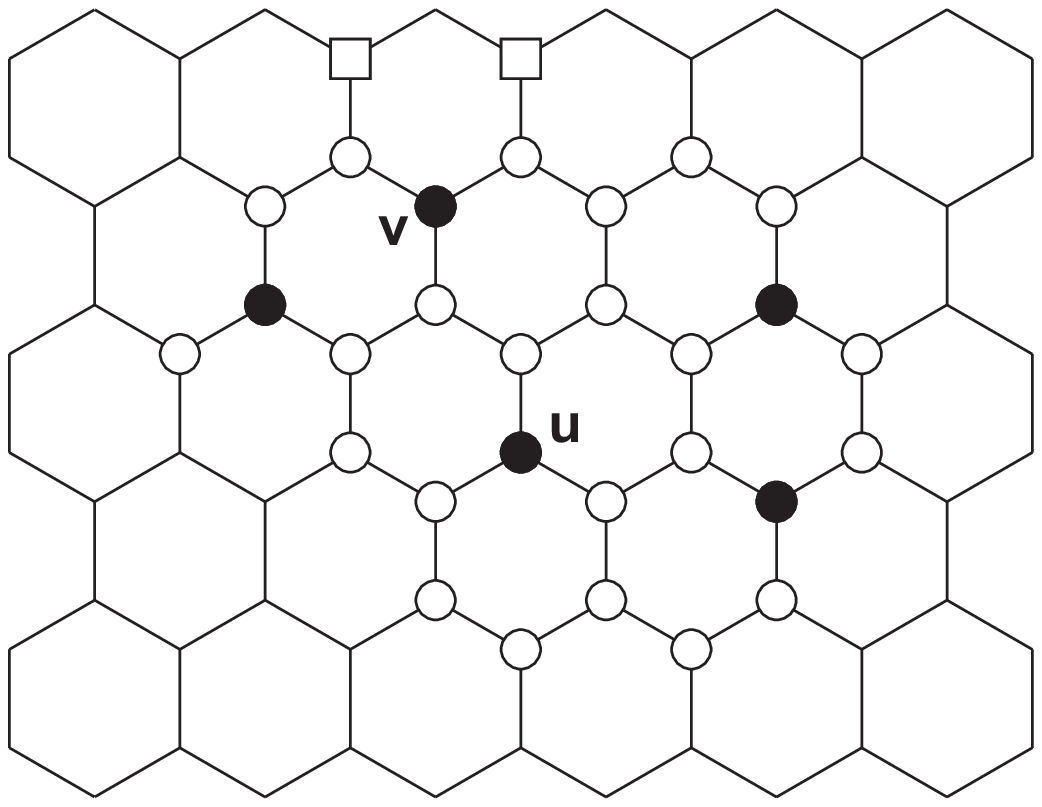}
\label{HexR2Rule8}
}
\subfigure[Rule~9]{
\includegraphics[height=80pt]{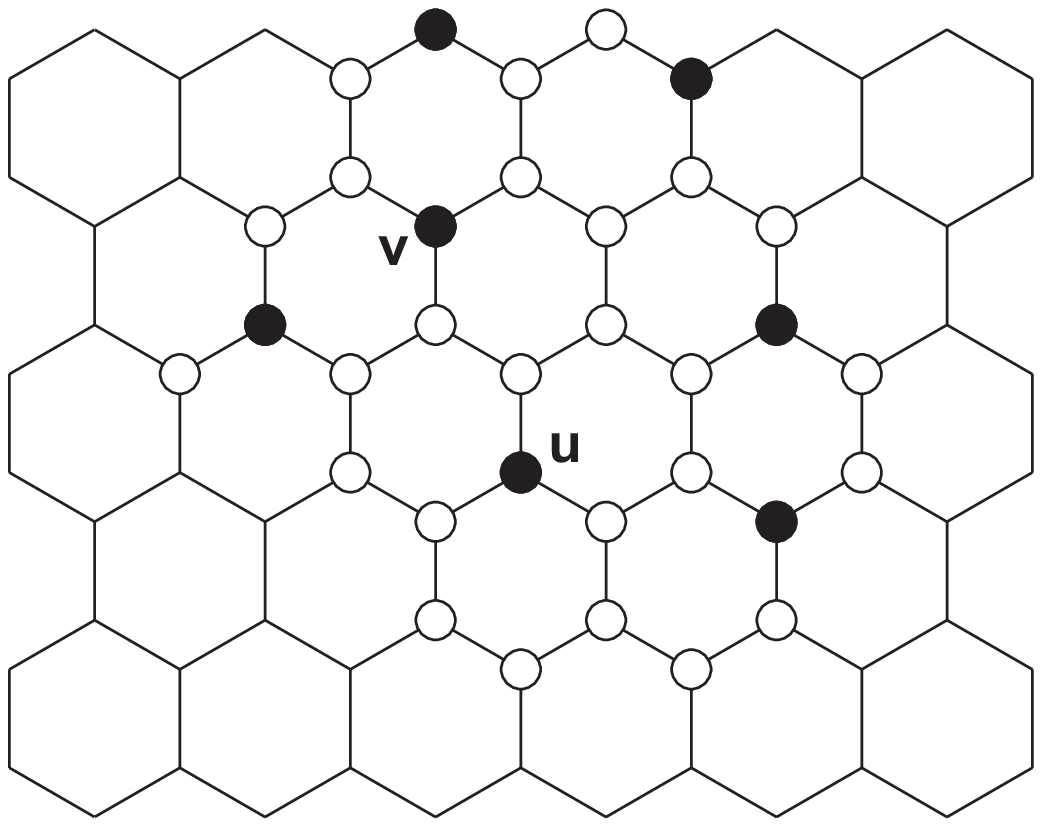}
\label{HexR2Rule9}
}
\subfigure[Rule~10]{
\includegraphics[height=80pt]{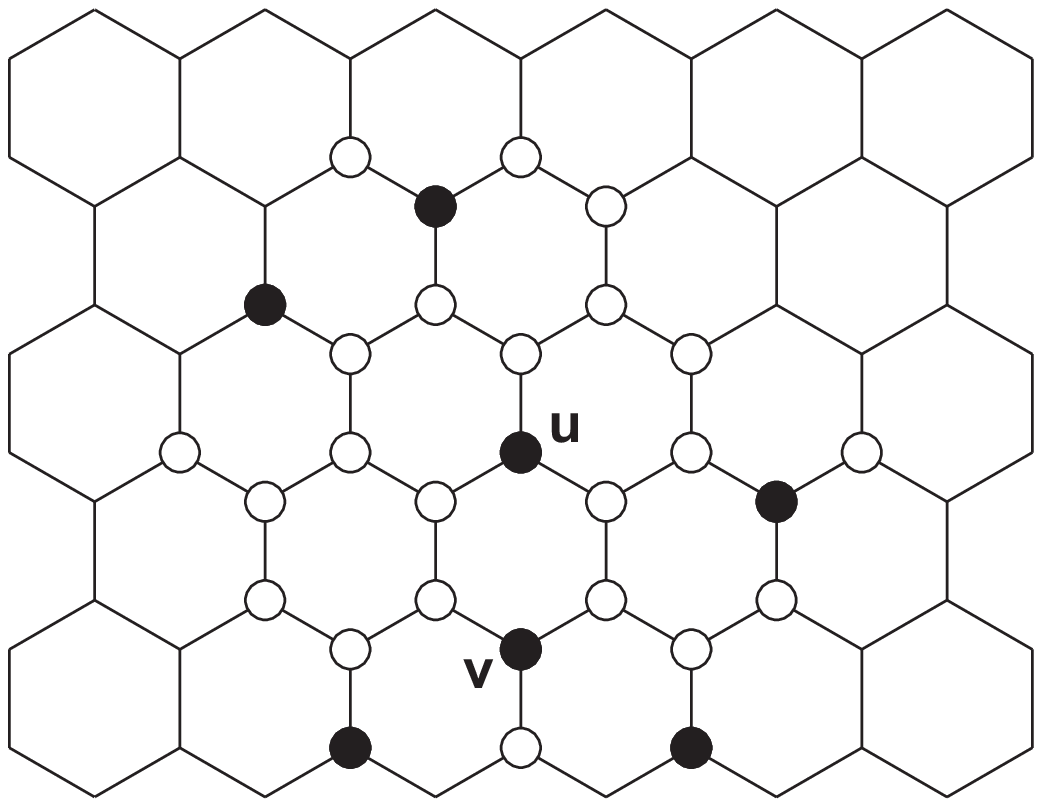}
\label{HexR2Rule10}
}
\caption{The rules of the shifting scheme illustrated. The black dots represent codewords and the white dots represent non-codewords. In the rules~7 and 8, at least one of the vertices marked with a white square is a codeword.}
\label{HexR2RuleFigure}
\end{figure}

The modified share of a codeword $\bc \in C$, which is obtained after the shifting scheme is applied, is denoted by $\ms_2(\bc)$. The use of the rules is illustrated in the following example.
\begin{figure}
\centering
\includegraphics[height=95pt]{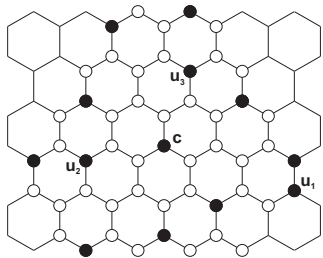}
\caption{An example of the use of the shifting rules.} 
\label{ShiftingRulesExampleFig}
\end{figure}
\begin{example}
Consider the codeword $\bc$ with the surroundings as illustrated in Figure~\ref{ShiftingRulesExampleFig}. The share of the codeword $\bc$ is equal to $5$. The rules 6, 8 and 9 apply to the codeword $\bc$ and according to the rules $1/12$ units of share is shifted from $\bc$ to $\bu_1$, $\bu_2$ and $\bu_3$, respectively. (Recall that reflections and rotations can be applied to the constellations in Figure~\ref{HexR2RuleFigure}.) Hence, after the shifting scheme is applied, we have $\ms_2(\bc) = 19/4$ for the modified share. In order to ensure that also $\ms_2(\bu_i) \leq 19/4$ for any $i = 1,2,3$, we refer to the proofs of Lemmas~\ref{HexR2ReceiveLemma2} and \ref{HexR2ReceiveLemma1}.
\end{example}

\subsection{The main theorem} \label{SubsectionMainTheorem}

The following three lemmas show that $\ms_2(\bc) \leq 19/4$ for all $\bc \in C$. The proofs of the lemmas are postponed to Section~\ref{SubsectionLemmasProofs}.

\begin{lemma} \label{HexR2ReceiveLemma2}
Let $\bc \in C$ be a codeword such that $\bc$ is not adjacent to another codeword and share is shifted to $\bc$ according to the previous rules. Then we have $\ms_2(\bc) \leq 19/4$.
\end{lemma}

\begin{lemma} \label{HexR2ReceiveLemma1}
Let $\bc \in C$ be a codeword such that $\bc$ is adjacent to another codeword and share is shifted to $\bc$ according to the previous rules. Then we have $\ms_2(\bc) \leq 19/4$.
\end{lemma}

\begin{lemma} \label{HexR2ReceivingLemma}
Let $\bc \in C$ be a codeword such that no share is shifted to $\bc$
according to the previous rules. Then we have $\ms_2(\bc) \leq 19/4$.
\end{lemma}

As stated in the previous lemmas, we have  $\ms_2(\bc) \leq 19/4$ for any $\bc \in C$. Now we are ready to prove the main theorem of the
paper.
\begin{thm} \label{HexR2MainThm}
If $C$ is a $2$-identifying code in the hexagonal grid $G_H$, then
the density
\[
D(C) \geq \frac{4}{19} \textrm{.}
\]
\end{thm}
\begin{proof}
Assume that $C$ is a $2$-identifying code in $G_H$. Since each vertex $\bu \in Q_{n-2}$ with $|I_2(\bu)|=i$ contributes the summand $1/i$
to $s_2(\bc)$ for each of the $i$ codewords $\bc \in B_2(\bu)$, we have
\begin{equation} \label{HexR2MainEq1}
\sum_{\bc \in C \cap Q_n} s_2(\bc) \geq |Q_{n-2}| \textrm{.}
\end{equation}
Furthermore, we have
\begin{equation} \label{HexR2MainEq2}
\sum_{\bc \in C \cap Q_n} s_2(\bc) \leq \sum_{\bc \in C \cap Q_n} \ms_2(\bc) + \frac{19}{4} |Q_{n+6} \setminus Q_n| \textrm{,}
\end{equation}
because shifting shares inside $Q_n$ does not affect the sum and each codeword in $Q_{n+6} \setminus Q_n$ can receive at most $19/4$ units
of share (by Lemmas~\ref{HexR2ReceiveLemma2} and \ref{HexR2ReceiveLemma1}). Notice also that
codewords in $Q_n$ cannot shift share to codewords outside $Q_{n+6}$. Therefore, combining the equations~\eqref{HexR2MainEq1} and
\eqref{HexR2MainEq2} with the fact that $\ms_2(\bc) \leq 19/4$ for any
$\bc \in C$, we obtain
\[
\frac{|C \cap Q_n|}{|Q_n|} \geq \frac{4}{19} \cdot \frac{|Q_{n-2}|}{|Q_n|} - \frac{|Q_{n+6} \setminus Q_n|}{|Q_n|} \textrm{.}
\]
Since $|Q_k| = (2k+1)^2$ for any positive integer $k$, it is
straightforward to conclude from the previous inequality that the
density $D(C) \geq 4/19$.
\end{proof}

\subsection{The proofs of the lemmas} \label{SubsectionLemmasProofs}

In what follows, we provide the proofs of Lemmas~\ref{HexR2ReceiveLemma2}, \ref{HexR2ReceiveLemma1} and \ref{HexR2ReceivingLemma}.

\begin{proof}[Proof of Lemma~\ref{HexR2ReceiveLemma2}]
Notice first that share can be shifted to the codeword $\bc$ only according to the rules 3, 7, 7.1, 8, 9 and 10 since $\bc$ is not adjacent to another codeword. The main idea of the following proof is to show that $\bc$ cannot receive share according to two different rules. In each case, this observation then quite straightforwardly implies the claim.

Assume first that $\bc$ receives share according to the rule 10. Without loss of generality, we may assume that the rule is applied as in Figure~\ref{HexR2Rule10} (when $\bc = \bv$). Now the vertices $\bc + (-2,-1)$, $\bc + (0,1)$ and $\bc + (2,-1)$ belong to the code $C$. Since this is not the case with the other rules (see Figure~\ref{HexR2RuleFigure}), they cannot be applied to $\bc$.  
Moreover, choosing $D = \{ \bc, \bc + (-2,-1), \bc + (0,1), \bc + (2,-1) \}$ in Lemma~\ref{SimpleEstLemma}, we have $s_2(\bc) \leq 9/2$. Thus, since according to the rule 10 share can be shifted to $\bc$ only from $\bc + (-2,-1)$, $\bc + (0,1)$ and $\bc + (2,-1)$ and at most once for each of these codewords, we obtain that $\ms_2(\bc) \leq s_2(\bc) + 3\cdot1/12 \leq 19/4$.

Assume then that share is shifted to $\bc$ according to the rule 9. First of all, by the previous paragraph, the rule 10 cannot be applied to $\bc$. Since now we have $I_2(\bc) = \{ \bc \}$, it is immediate that $\bc$ cannot receive share according to the rules 3, 7, 7.1 or 8. Moreover, it is easy to see that the rule 9 can be applied only once. Therefore, since we have $s_2(\bc) \leq 14/3$ by Lemma~\ref{SimpleEstLemma}, we obtain that $\ms_2(\bc) \leq s_2(\bc) + 1/12 \leq 19/4$.

Assume that $\bc$ receives share according to the rule 8 and that the rule is used as in Figure~\ref{HexR2Rule8}. Now, if the rule 7 or 7.1 was used, then there would exist two codewords in $B_2(\bc)$ such that the distance between them is equal to $4$. By the constellation of the rule 8, this is impossible. Let us then show that neither the rule 3 can be used. Assume to the contrary that share is shifted to $\bc$ according to the rule 3. Since there is a pair of adjacent codewords in the constellation of the rule 3, the vertex $\bc + (0,1)$ belongs to $C$. Furthermore, we have either $\bc + (-1,1) \in C$ or $\bc + (1,1) \in C$ (but not both). If $\bc + (-1,1) \in C$, then we have a contradiction since share is shifted from $\bc + (3,0)$, which is not a codeword in the constellation of the rule 8. On the other hand, if $\bc + (1,1) \in C$, then share is received from $\bc + (-3,0)$. This again leads to a contradiction since $\bc + (-2,-1) \in C$. In conclusion, only the rule 8 can be applied to $\bc$. Moreover, the rule 8 can be used at most once. If $\bc + (-1,1) \in C$ or $\bc + (1,1) \in C$, then $s_2(\bc) \leq 55/12$ or $s_2(\bc) \leq 23/6$ by Lemma~\ref{SimpleEstLemma}, respectively. Hence, we have $\ms_2(\bc) \leq s_2(\bc) + 1/12 \leq 56/12 \leq 19/4$.

Assume that $\bc$ receives share according to the rule 3 and that the rule is used as in Figure~\ref{HexR2Rule3} (when $\bc = \bv'$). By the previous considerations, we know that share cannot be shifted to $\bc$ according to the rules 8, 9 and 10. Let us then show that neither the the rules 7 or 7.1 can be applied to $\bc$. Assume to the contrary that $\bc$ receives share according to the rule 7. Now $\bc$ can receive share only from the vertices $\bc + (-1,-2)$, $\bc + (1,-2)$, $\bc + (3,0)$, $\bc + (2,1)$, $\bc + (-3,0)$ and $\bc + (-2,1)$. Now we have the following observations:
\begin{itemize}
\item Since $\bc + (-1,-2) \in C$ share cannot be shifted from $\bc + (-1,-2)$ and $\bc + (1,-2)$.
\item Since $\bc + (-2,0) \in C$ share cannot be shifted from $\bc + (-3,0)$ and $\bc + (-2,1)$.
\item Since $\bc + (-1,-1) \notin C$ and $\bc + (1,-1) \notin C$ share cannot be shifted from $\bc + (3,0)$ and $\bc + (2,1)$.
\end{itemize}
A contradiction now follows from these facts. Hence, only the rule 3 can be applied to $\bc$. Moreover, it is easy to see that share can be shifted to $\bc$ at most twice according to the rule 3. Therefore, since we have $s_2(\bc) \leq  9/2$ by Lemma~\ref{SimpleEstLemma}, we obtain that $\ms_2(\bc) \leq s_2(\bc) + 2 \cdot 1/12 \leq 19/4$.

Finally, assume that $\bc$ receives share according to the rule 7 or 7.1 and that the rule is used as in Figure~\ref{HexR2Rule7}. As shown above, other rules cannot be applied to $\bc$. Moreover, the rules 7 and 7.1 can be applied to $\bc$ only once. Furthermore, if $\bc+(2,0) \in C$ or $\bc+(1,1) \in C$, then we have $s_2(\bc) \leq 23/6$ or $s_2(\bc) \leq 53/12$ by Lemma~\ref{SimpleEstLemma}, respectively. Thus, we have $\ms_2(\bc) \leq s_2(\bc) + 1/4 \leq 14/3 \leq 19/4$.
\end{proof}

\begin{figure}
\centering
\subfigure[The first case]{
\includegraphics[height=90pt]{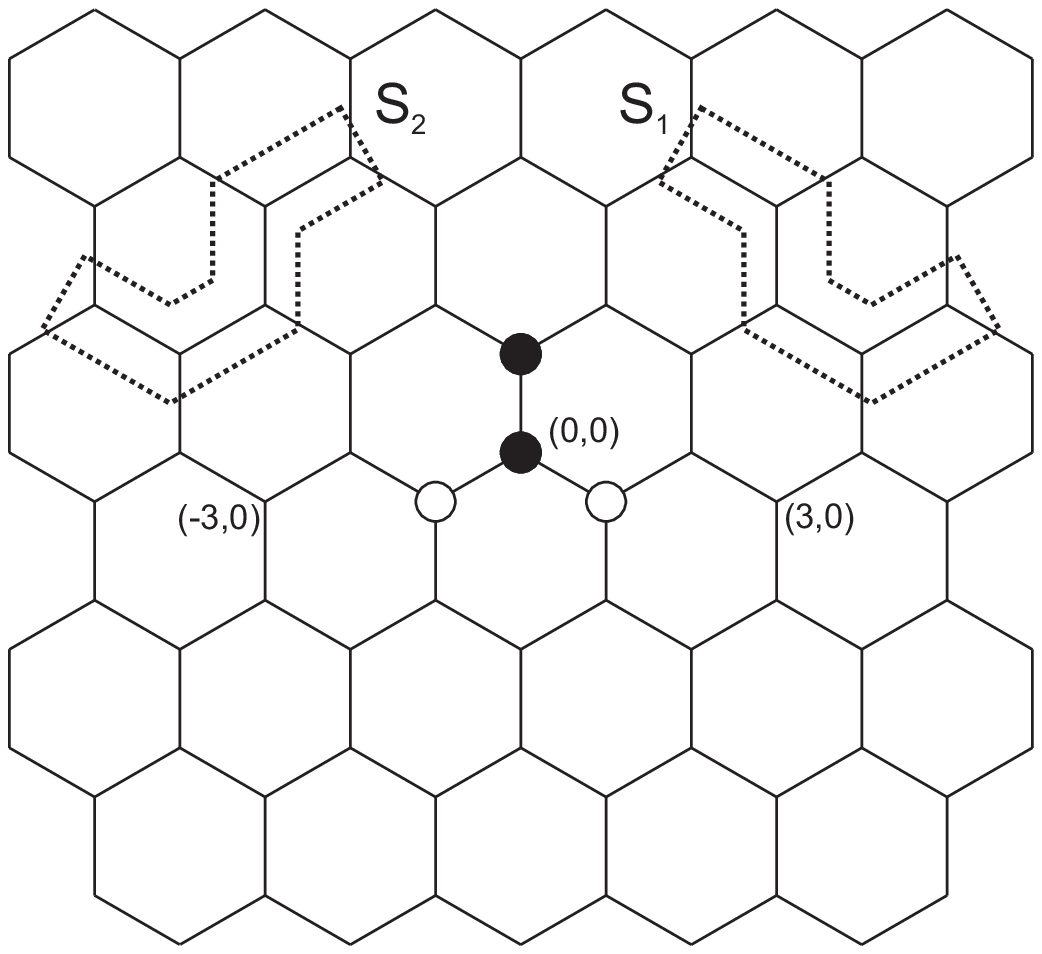}
\label{Lemma1CaseA}
}
\subfigure[The second case]{
\includegraphics[height=90pt]{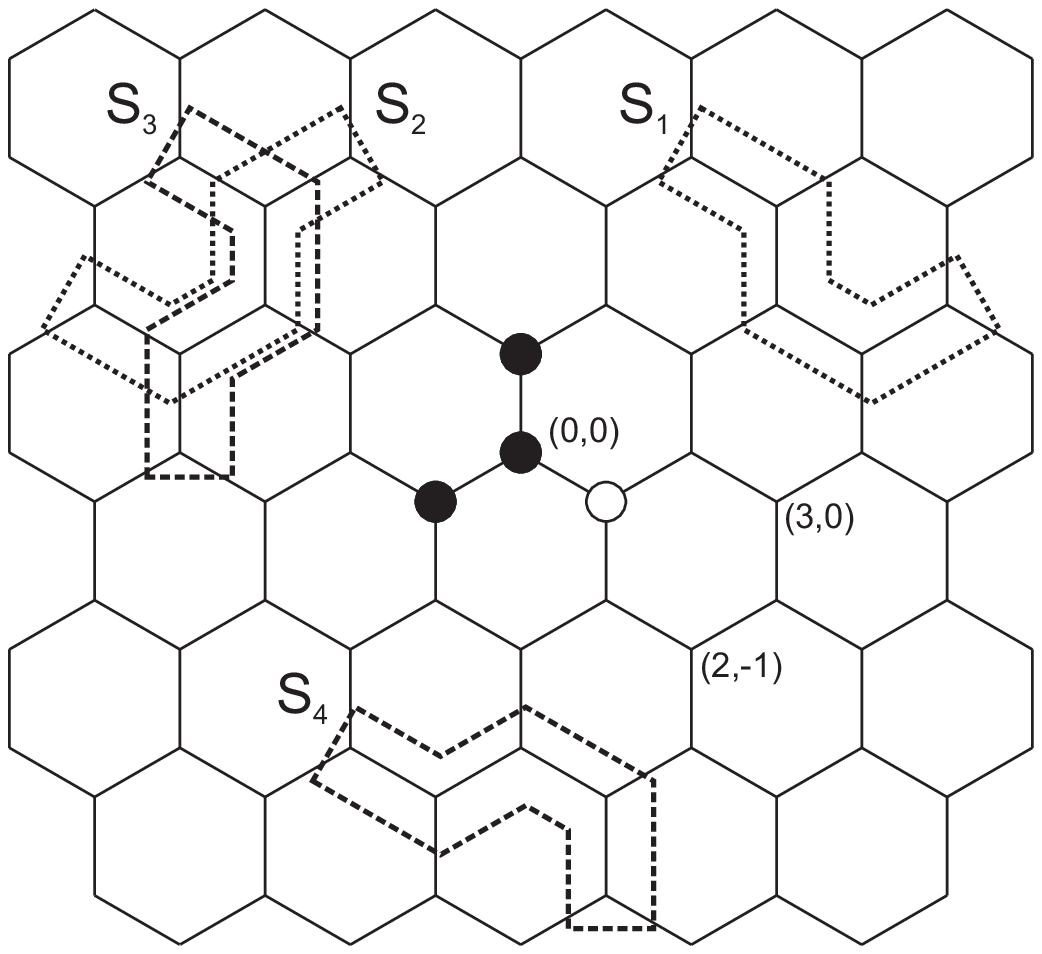}
\label{Lemma1CaseB}
}
\subfigure[The third case]{
\includegraphics[height=90pt]{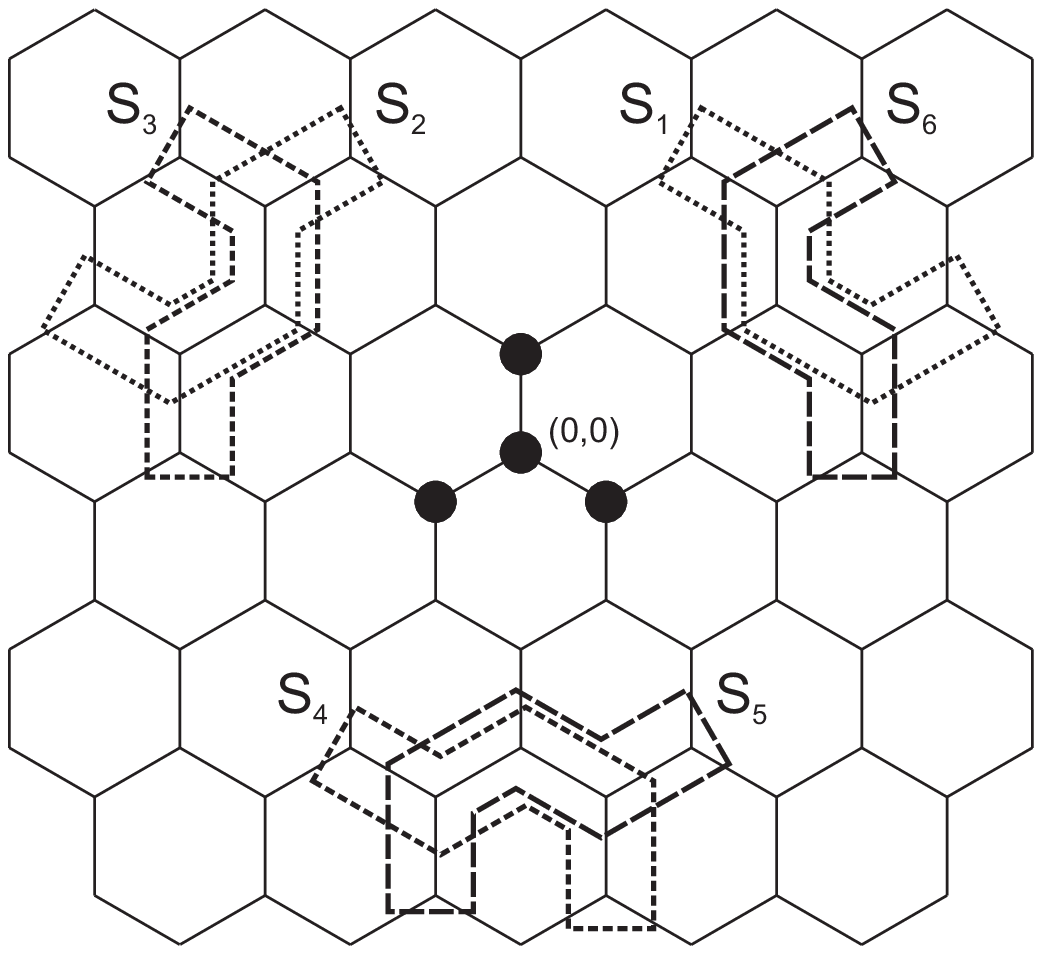}
\label{Lemma1CaseC}
}
\caption{The cases of the proof of Lemma~\ref{HexR2ReceiveLemma1} illustrated.} 
\label{Lemma1Cases}
\end{figure}

\begin{proof}[Proof of Lemma~\ref{HexR2ReceiveLemma1}]
Since $\bc$ is adjacent to another codeword, it is immediate that $\bc$ can receive share only according to the rules 1, 1.1, 1.2, 2, 2.1, 2.2, 2.3, 3, 4, 5 and 6. 
The proof of the lemma is now divided into three cases depending on the number of codewords adjacent to $\bc$.

\medskip

1) Assume first that $\bc$ is adjacent to exactly one codeword. Without loss of generality, we may assume that $\bc = (0,0)$ and the adjacent codeword is $(0,1)$. (Notice that $(-1,0) \notin C$ and $(1,0) \notin C$.) Now the only possibilities for $\bc$ to receive share is from the vertices $(-3,0)$ or $(3,0)$ (the rule 1.2) and from the vertices that belong to $S_1 = \{ (5,1), (4,1), (3,1), (3,2), (2,2)\}$ and $S_2 = \{ (-5,1), (-4,1), (-3,1), (-3,2), (-2,2)\}$. These observations are illustrated in Figure~\ref{Lemma1CaseA}.

Consider then more closely the set $S_1$. In what follows, we show that at most $1/4$ units of share is shifted from the vertices of $S_1$ to $\bc$. Notice that if share is shifted only from one vertex of $S_1$, then we are immediately done. Observe then that if $\bu \in S_1$ shifts share to $\bc$ according to the rules 1--6 or their modifications, then we have $I_2(\bu) = \{ \bu \}$. Therefore, if two vertices of $S_1$ shift share, then one of these vertices is $(4,1)$ or $(5,1)$.
Assume that share is shifted from $(4,1)$ according to the rule 5 (see Figure~\ref{HexR2Rule5}). Now we have $C \cap S_1 = \{ (2,2), (4,1) \}$ and $(4,3) \in C$. Therefore, since at most $1/12$ units of share is shifted from $(4,1)$ according to the rule 2.1, we obtain that no more than $1/6+1/12 = 1/4$ units of share can be shifted from $S_1$ to $\bc$. Similarly, if share is shifted from $(5,1)$ according to the rule 6, then it can be shown that $\bc$ receives at most $1/4$ units of share from $S_1$. In conclusion, at most $1/4$ units of share is shifted from the vertices of $S_1$ to $\bc$. Analogously, this statement also holds for the vertices of $S_2$.

Assume that the rule 1.2 is used. (Clearly, this rule can be used only once.) Without loss of generality, we may assume that $(-3,0) \in C$ and $(1,-1) \in C$. Therefore, choosing $D= \{ \bc, (0,1), (-3,0), (1,-1) \}$ in Lemma~\ref{SimpleEstLemma}, we obtain that $s_2(\bc) \leq 15/4$. Thus, since the codewords in each of the sets $S_1$ and $S_2$ can shift at most $1/4$ units of share to $\bc$, we have $\ms_2(\bc) \leq s_2(\bc) + 3 \cdot 1/4 \leq 9/2$. Assume then that the rule 1.2 cannot be applied to $\bc$. Since $\bc$ and $(0,1)$ are $2$-separated by $C$, there exists at least one codeword in the symmetric difference $B_2(\bc) \, \triangle \, B_2(0,1)$. Thus, without loss of generality, we may assume that $(1,-1) \in C$, $(1,2) \in C$, $(2,0) \in C$ or $(2,1) \in C$. The first part of the proof is then concluded by the following four cases:
\begin{itemize}
\item Assume that $(1,-1) \in C$. Now, by the first case of Example~\ref{IllExampleAdjacentLemma}, we know that $s_2(\bc) \leq 17/4$. Therefore, since share is shifted to $\bc$ only from the vertices of the sets $S_1$ and $S_2$, we have $\ms_2(\bc) \leq s_2(\bc) + 2 \cdot 1/4 \leq 19/4$.
\item Assume that $(1,2) \in C$. It is straightforward to verify that share can be shifted to $\bc$ only from $(-3,1)$, $(-2,2)$, $(-3,2)$, $(-4,1)$ and $(-5,1)$ according to the rules 1.1, 2.3, 3, 5 and 6, respectively. Moreover, it is easy to see that at most one of these rules can be used (and only once). Thus, the codeword $\bc$ receives at most $1/6$ units of share. Hence, we have $\ms_2(\bc) \leq s_2(\bc) + 1/6 \leq 19/4$ since $s_2(\bc) \leq 55/12$ by the first case of Example~\ref{IllExampleAdjacentLemma} (and obvious symmetrical argument).
\item Assume that $(2,0) \in C$. Consider then more closely the vertices of $S_1$. Then it is easy to conclude that the vertices $(5,1)$, $(4,1)$ and $(3,1)$ cannot shift share to $\bc$. Hence, only either $(3,2)$ according to the rule 3 or $(2,2)$ according to the rule 2.2 (but not both) is capable of shifting share to $\bc$. In both cases, $\bc$ receives at most $1/6$ units of share. By the second case of Example~\ref{IllExampleAdjacentLemma}, we know that $s_2(\bc) \leq 13/3$. Therefore, since at most $1/4$ units of share is received from $S_2$, we have $\ms_2(\bc) \leq s_2(\bc) + 1/4 + 1/6 \leq 19/4$.
\item Assume that $(2,1) \in C$. By the second case of Example~\ref{IllExampleAdjacentLemma} (and symmetry), we obtain that $s_2(\bc) \leq 9/2$. Since now share can be shifted to $\bc$ only from $S_2$, we obtain that $\ms_2(\bc) \leq s_2(\bc) + 1/4 \leq 19/4$.
\end{itemize}

\medskip

2) Assume that $\bc$ is adjacent to exactly two codewords. Without loss of generality, we may assume that $\bc = (0,0)$ and that the adjacent codewords are $(-1,0)$ and $(0,1)$. Let then $S_3$ and $S_4$ be sets which are obtained by rotating respectively the sets $S_1$ and $S_2$ by $2\pi/3$ (counter-clockwise in the honeycomb representation) around the origin. Again the codewords in each of these sets $S_i$ can shift at most $1/4$ units of share to $\bc$. In addition to the previous ones, at most $1/4$ units of share can also be shifted to $\bc$ from either $(2,-1)$ or $(3,0)$ (but not both) according to the rule 1.2. These observations are illustrated in Figure~\ref{Lemma1CaseB}.

Assume first that share is shifted to $\bc$ according to the rule 1.2. Without loss of generality, we may assume that $\bc$ receives share from the vertex $(3,0)$. Then we immediately have $(3,0) \in C$ and $(-1,-1) \in C$. Therefore, as above, we have $s_2(\bc) \leq 15/4$ by Lemma~\ref{SimpleEstLemma}. Furthermore, since $(-1,-1) \in C$, the codeword $\bc$ does not receive share from the set $S_4$. Thus, we have $\ms_2(\bc) \leq s_2(\bc) + 4 \cdot 1/4 \leq 19/4$.

Assume then that the rule 1.2 is not used. Since the vertices $\bc$ and $(0,1)$ are $2$-separated by $C$, there exists a codeword in $B_r(\bc) \, \triangle \, B_r(0,1)$. If $(1,-1) \in C$ or $(2,0) \in C$, then $s_2(\bc) \leq 53/15$ (by Lemma~\ref{SimpleEstLemma}) and we are immediately done since at most $1$ unit of share can be shifted to $\bc$ from the union of the sets $S_i$. If $(-2,1) \in C$, then $s_2(\bc) \leq 21/5$ and we are done since share cannot be shifted to $\bc$ from $S_2$ and $S_3$. If $(-1,-1) \in C$ or $(-2,0) \in C$, then $s_2(\bc) \leq 79/20$ (by Lemma~\ref{SimpleEstLemma}) and we are again done (since share is not shifted from $S_3$). Hence, we may assume that $(-1,2) \in C$, $(1,2) \in C$ or $(2,1) \in C$. Analogously, it can also be assumed that $(-3,0) \in C$, $(-2,-1) \in C$ or $(0,-1) \in C$ since $\bc$ and $(-1,0)$ are $2$-separated by $C$. Thus, at most two of the sets $S_i$ can shift share to $\bc$. Therefore, since $s_2(\bc) \leq 17/4$ (choose $D = \{ \bc, (-1,0), (0,1) \}$ in Lemma~\ref{SimpleEstLemma}), we have $\ms_2(\bc) \leq s_2(\bc) + 2 \cdot 1/4 \leq  19/4$.

\medskip

3) Finally, assume that all the vertices adjacent to $\bc$ are codewords, i.e. $(-1,0) \in C$, $(0,1) \in C$ and $(1,0) \in C$ (see Figure~\ref{Lemma1CaseC}). Notice that now the rule 1.2 cannot be used. Since there again exists a codeword in $B_r(\bc) \, \triangle \, B_r(0,1)$, it is easy to conclude as above that at most $5 \cdot 1/4$ units of share is shifted to $\bc$. Therefore, since $s_2(\bc) \leq 67/20$ by Lemma~\ref{SimpleEstLemma}, we have $\ms_2(\bc) \leq s_2(\bc) + 5/4 \leq 19/4$. This completes the proof of the lemma.
\end{proof}

\begin{proof}[Proof of Lemma~\ref{HexR2ReceivingLemma}]
Without loss of generality, we may assume that $\bc=(0,0)$. Assume
first that $|I_2(\bc)| \geq 2$. If now $\bc$ is adjacent to another
codeword, then $\ms_2(\bc) \leq 19/4$ (by Lemma~\ref{SimpleEstLemma}). Hence, we may assume that $(-1,0), (1,0), (0,1) \notin C$. Let then $(2,0)$ be a codeword of $C$. Since the vertices $\bc$ and $(1,0)$ are $2$-separated by C, there is at least one codeword in the symmetric difference $B_2(\bc) \, \triangle \, B_2(1,0)$ (see Figure~\ref{Lemma3CodewordsD2}). Therefore, by Lemma~\ref{SimpleEstLemma}, it is straightforward (albeit tedious) to verify that in all possible cases $\ms_2(\bc) \leq s_2(\bc) \leq 19/4$. Indeed, we can choose in Lemma~\ref{SimpleEstLemma} the set $D$ to consist of $\bc$, $(2,0)$ and a codeword in the symmetric difference $B_2(\bc) \, \triangle \, B_2(1,0)$.

\begin{figure}
\centering
\includegraphics[height=100pt]{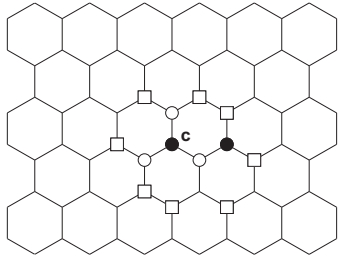}
\caption{The symmetric difference $B_2(\bc) \, \triangle \, B_2(1,0)$ consists of the squared vertices. One of these vertices is a codeword.} 
\label{Lemma3CodewordsD2}
\end{figure}

From now on, we may assume that $I_2(\bc) = \{ \bc \}$. In what follows, we use the notations: $A_1 = \{ (-1,1), (0,1), (1,1) \}$, $A_2 = \{ (-2,0), (-1,0), (-1,-1) \}$, $A_3 = \{ (1,-1), (1,0), (2,0) \}$, $A'_1 = \{ (-1,2), (1,2) \}$, $A'_2 = \{ (-3,0), (-2,-1) \}$ and $A'_3 = \{ (2,-1), (3,0) \}$. These sets are illustrated in Figure~\ref{Lemma3ASetsIllustrated}. The proof of the lemma now divides into three cases depending on the number of codewords in the set $\{ (-2,1), (2,1), (0,-1) \}$

\medskip

1) Assume first that $(-2,1), (2,1), (0,-1) \notin C$. Since the vertices $\bc$, $(-1,0)$, $(1,0)$ and $(0,1)$ are $2$-separated by $C$, each one of the sets $A'_1$, $A'_2$ and $A'_3$ contains at least one codeword. Hence, each of the sets $A_1$, $A_2$ and $A_3$ contains a vertex whose $I$-set contains at least three codewords. Indeed, if for example $(-1,2) \in C$, then $(-1,1)$ or $(0,1)$ is such a vertex in $A_1$. Therefore, we obtain that $s_2(\bc) \leq 1+6 \cdot 1/2+3 \cdot 1/3=5$.

\begin{figure}
\centering
\includegraphics[height=100pt]{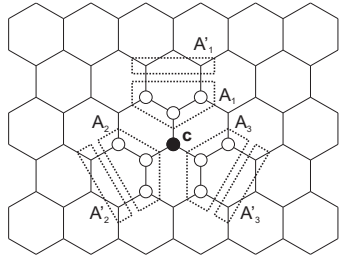}
\caption{The sets $A_1$, $A_2$, $A_3$, $A'_1$, $A'_2$ and $A'_3$ illustrated.} 
\label{Lemma3ASetsIllustrated}
\end{figure}

Consider then the set $A'_1$ that contains at least one codeword as stated above. Assume first that both $(-1,2) \in C$ and $(1,2) \in C$. If the vertex $(0,2)$ also belongs to $C$, then the $I$-sets of all the vertices in $A_1$ have size at least $3$. Then there are at least $5$ vertices in $B_2(\bc)$ which are $2$-covered by at least $3$ codewords. Therefore, we have $\ms_2(\bc) \leq s_2(\bc) \leq 1 + 4 \cdot 1/2 + 5 \cdot 1/3 = 14/3$ (and we are done). Hence, suppose that $(0,2) \notin C$. Furthermore, assume first that $(-2,2) \in C$. If also $(-3,1) \in C$, then $|I_2(0,1)| \geq 3$ and $|I_2(-1,1)| \geq 4$. Thus, we have $\ms_2(\bc) \leq s_2(\bc) \leq 1 + 5 \cdot 1/2 + 3 \cdot 1/3 + 1/4 = 19/4$. On the other hand, if $(-3,1) \notin C$, then the rule 2.3 can be applied to $\bc$ and we obtain that $\ms_2(\bc) \leq s_2(\bc) - 1/12 \leq 29/6 - 1/12 \leq 19/4$. Hence, we may assume that $(-2,2) \notin C$ and $(-2,2) \notin C$ (by symmetry). Since the vertices $(-1,2)$ and $(1,2)$ are $2$-separated by $C$, at least one of the vertices $(-3,2)$, $(-2,3)$, $(2,3)$ and $(3,2)$ belongs to $C$. Hence, we can shift at least $1/4$ units of share from $\bc$ according to the rule 7. Therefore, we have $\ms_2(\bc) \leq s_2(\bc) - 1/4 \leq 19/4$.

By the considerations above, we may without loss of generality assume that $(-1,2) \in C$ and $(1,2) \notin C$. If $(0,2) \in C$, then $1/4$ units of share can be shifted from $\bc$ according to the rule 1 or 1.2, and we are done. Thus, suppose that $(0,2) \notin C$. Assume then that $(-2,2) \in C$. If the rule 2 applies to $\bc$, then we are immediately done ($1/4$ units of share is shifted from $\bc$). On the other hand, if $(-3,1) \in C$, then by the fact that $|I_2(-1,1)| \geq 4$ we have $s_2(\bc) \leq 59/12$ and we are again done since at least $1/6$ units of share is shifted from $\bc$ according to the rule 2.2. Hence, assume that $(-2,2) \notin C$. Since $(0,1)$ and $(-1,1)$ are $2$-separated by $C$, the vertex $(-3,1)$ belongs to $C$. Now, if $(-3,2) \in C$, then $1/4$ units of share can be shifted from $\bc$ according to the rule~3 and we are done. Thus, we may assume that $(-3,2) \notin C$. If $(-4,1) \in C$ and $(-3,0) \notin C$, then the rule 4 can be applied to $\bc$ and we are done. Furthermore, if $(-4,1) \in C$, and $(-3,0) \in C$, then instead of the set $A'_1$ consider the set $A'_3$. Repeating the previous arguments for the set $A'_3$, we obtain that $\ms_2(\bc) \leq 19/4$ also in this case. Thus, we may assume that $(-4,1) \notin C$. The surroundings of the codeword $\bc$ (obtained above) is illustrated in Figure~\ref{Lemma3LeftIllustrated}.

\begin{figure}
\centering
\subfigure[The first case]{
\includegraphics[height=100pt]{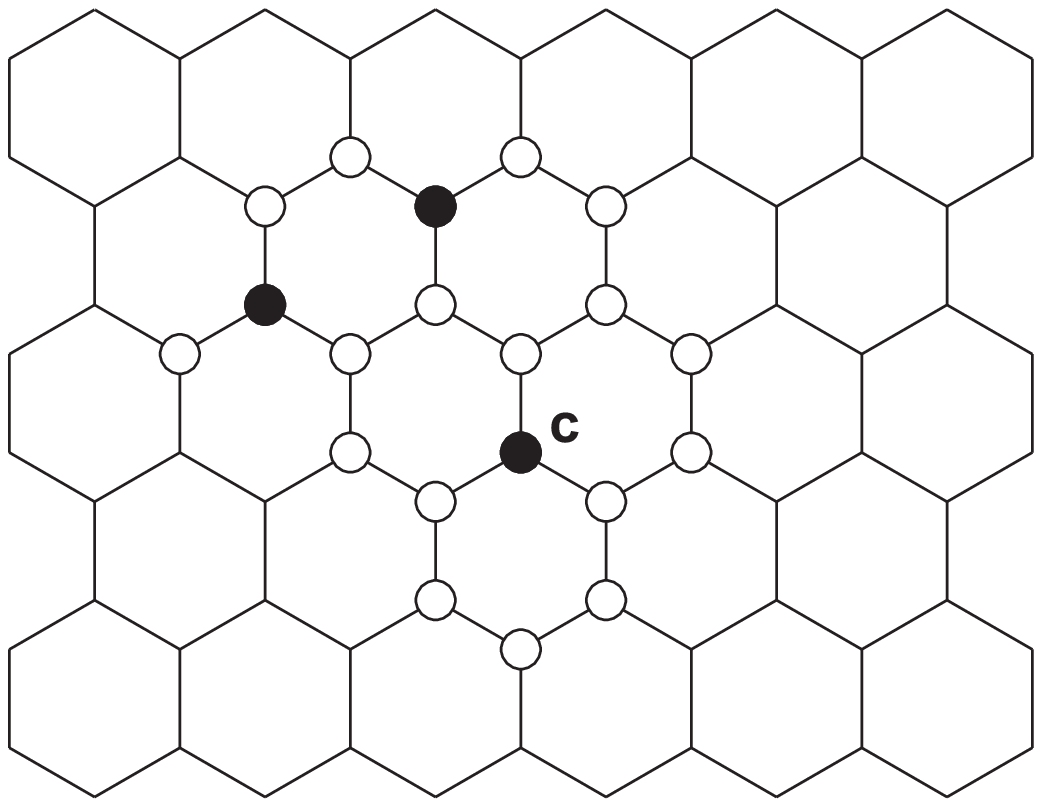}
\label{Lemma3LeftIllustrated}
}
\subfigure[The second case]{
\includegraphics[height=100pt]{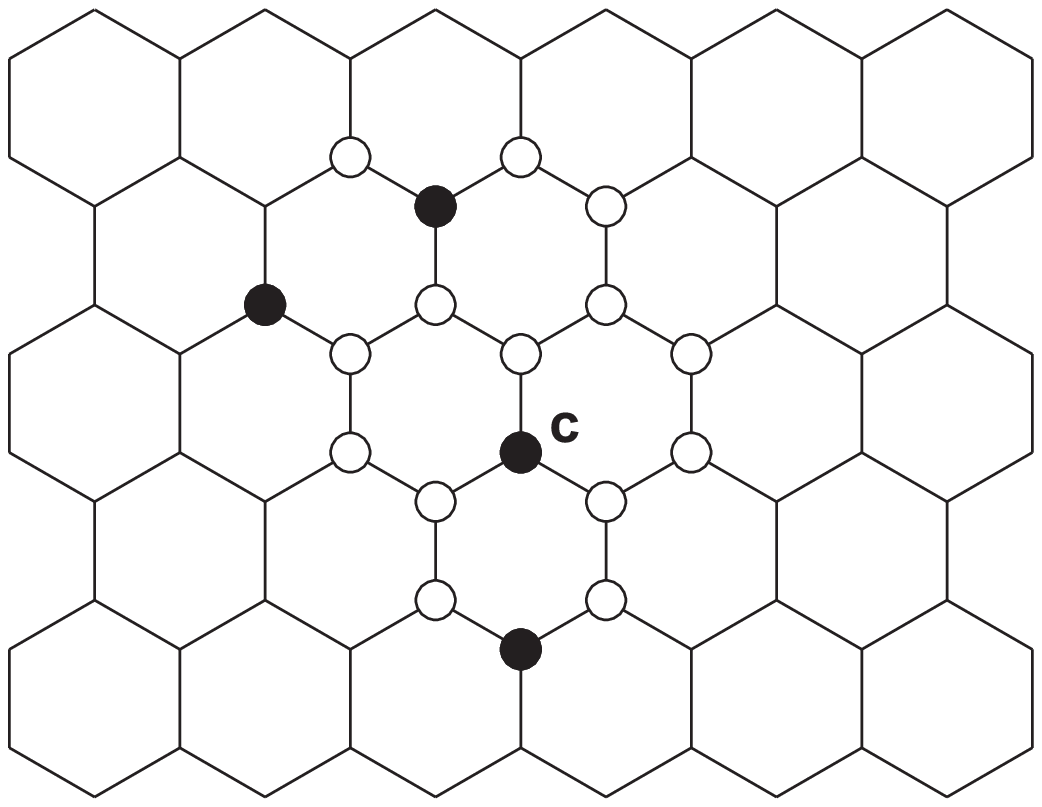}
\label{Lemma3UpDownIllustrated}
}
\caption{The surroundings of the codeword $\bc$ illustrated in (a) the first and (b) the second part of the proof of Lemma~\ref{HexR2ReceivingLemma}.} 
\label{HexR2SurroundingsIllustrated}
\end{figure}


The previous reasoning also applies when we consider the sets $A'_2$ and $A'_3$ instead of $A'_1$. This leads straightforwardly to the observation that we have only two different neighbourhoods of $\bc$ (up to rotations and reflections). These neighbourhoods are illustrated in Figure~\ref{HexR2ReceivingIllustrated}.

\begin{figure}
\centering
\subfigure[\!\!\!]{
\includegraphics[height=100pt]{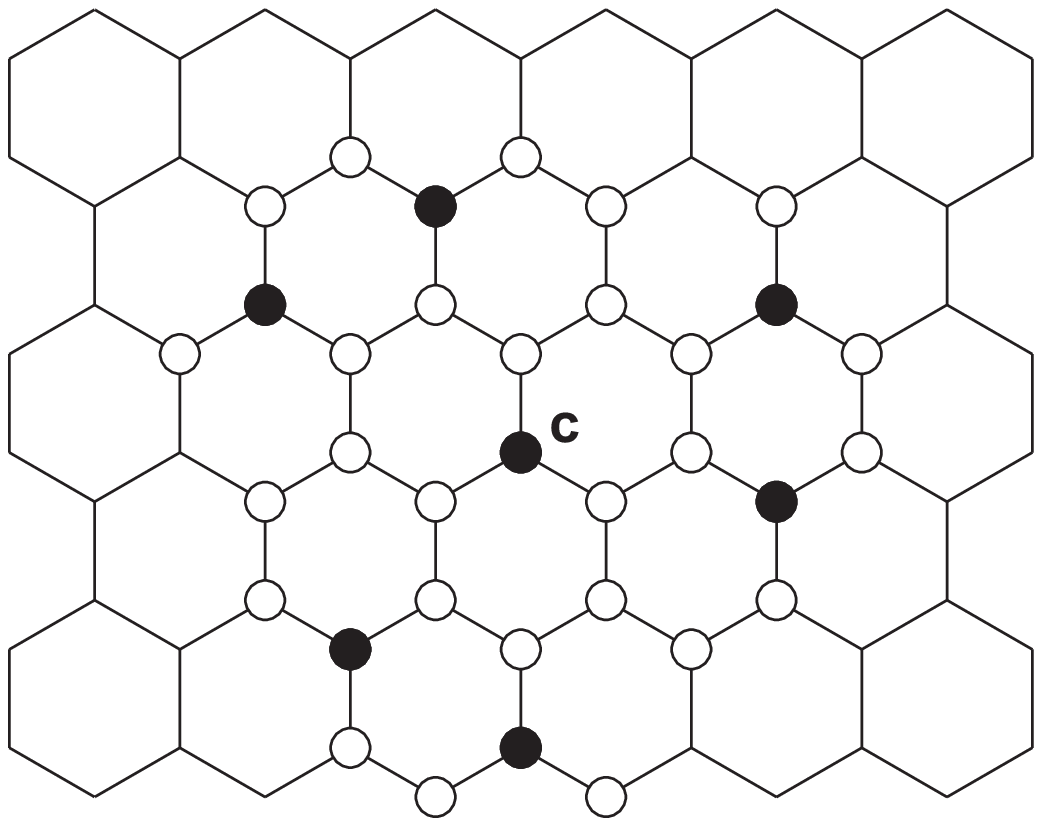}
\label{HexR2ReceivingIllustratedA}
}
\subfigure[\!\!\!]{
\includegraphics[height=100pt]{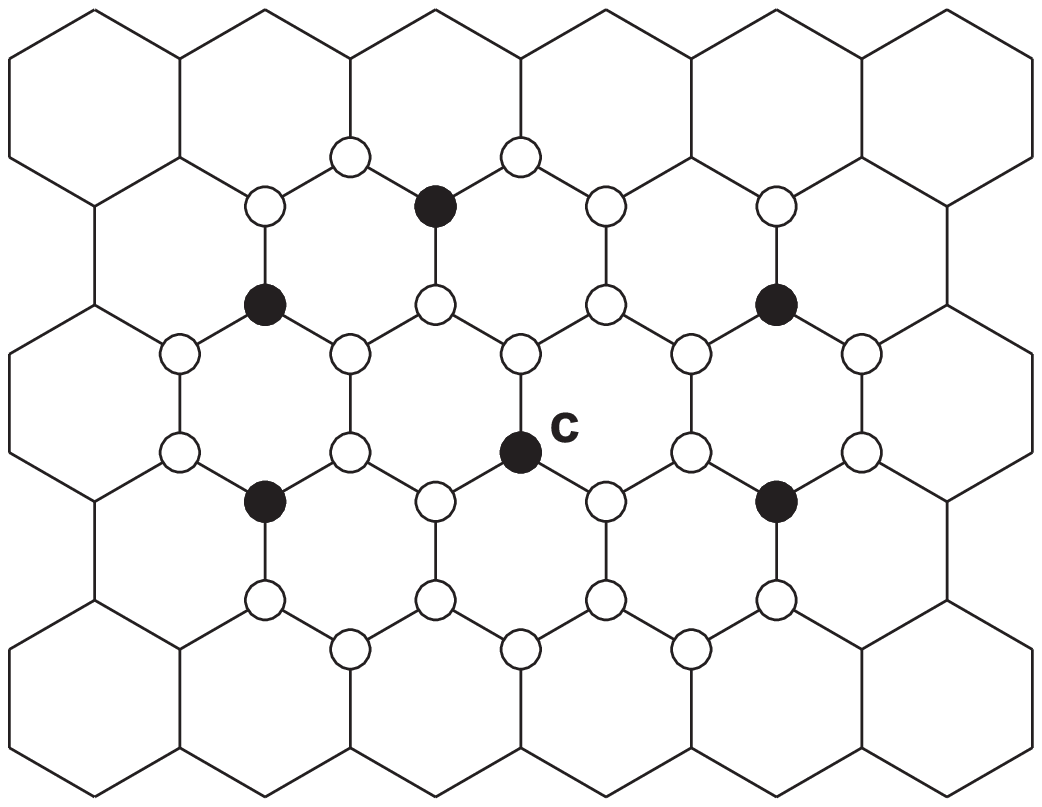}
\label{HexR2ReceivingIllustratedB}
}
\caption{Two cases of the proof of Lemma~\ref{HexR2ReceivingLemma} illustrated.} 
\label{HexR2ReceivingIllustrated}
\end{figure}

Consider first the case in Figure~\ref{HexR2ReceivingIllustratedA}. In what follows, we show that $\bc$ shifts $1/12$ units of share to $(-1,2)$ or $(1,3)$, or that we originally have at least $2$ vertices in $A_1$ such that their $I$-sets are at least of size $3$. In both cases, the (maximum) share of $\bc$ is reduced by at least $1/12$. (Actually, in the latter case, the share is reduced by $1/6$.) This observation can then be generalized to the (other two) symmetrical cases implying that $\ms_2(\bc) \leq 5-3\cdot 1/12 = 19/4$. If $(2,2) \in C$, then $(-1,1) \in A_1$ and $(1,1) \in A_1$ are $2$-covered by three codewords. Hence, we may assume that $(2,2) \notin C$. If we can shift $1/12$ units of share from $\bc$ to $(-1,2)$ according to the rule 8, then we are immediately done. Therefore, we may assume that $(-2,3) \notin C$ and $(0,3) \notin C$. Since $(-1,2)$ and $(1,2)$ are $2$-separated by $C$, the vertex $(2,3)$ belongs to $C$. Furthermore, since $(-1,2)$ and $(0,2)$ are $2$-separated by $C$, at least one of the vertices $(-1,3)$ and $(1,3)$ is a codeword. Therefore, at least $1/12$ units of share can shifted from $\bc$ according to the rule 6 or 9.

Consider then the case in Figure~\ref{HexR2ReceivingIllustratedB}.
Let us now show that $\bc$ shifts at least $1/6$ units of share to
$(-1,-2)$ or $(1,-2)$, or that we originally have more than $2$ vertices in $A_2 \cup A_3$ such that their $I$-sets are at least of size $3$. In both cases, the (maximum) share of $\bc$ is reduced by at least $1/6$. This result together with the observation in the previous paragraph, 
then implies that $\ms_2(\bc) \leq 5-1/6-1/12 = 19/4$. If now $(0,-2) \in C$, then we know that $(-2,0) \in A_2$, $(2,0) \in A_3$, and at least one of the vertices $(-1,-1) \in A_2$ and $(1,-1) \in A_3$ are such that their $I$-sets are at least of size $3$. Hence, we may assume that $(0,-2) \notin C$. Since $\bc$, $(-1,-1)$ and $(1,-1)$ are $2$-separated by $C$, the vertices $(-2,-2)$ and $(2,-2)$ belong to $C$. Furthermore, since $(0,-1)$ is $2$-covered by a codeword of $C$, we have $(-1,-2) \in C$ or $(1,-2) \in C$. Therefore, we can shift at least $1/6$ units of share to $(-1,-2)$ or $(1,-2)$ according to the rule 5. In conclusion, if $\bc$ is a codeword such that $I_2(\bc) = \{ \bc \}$ and $(-2,1), (2,1), (0,-1) \notin C$, then we have $\ms_2(\bc) \leq 19/4$.


\medskip

2) Assume then that $\bc$ is a codeword such that $I_2(\bc) = \{ \bc \}$, $(-2,1) \notin C$, $(2,1) \notin C$ and $(0,-1) \in C$. Now, by considering the vertices $(-1,-1) \in A_2$, $(-1,0) \in A_2$, $(1,-1) \in A_3$ and $(1,0) \in A_3$, we obtain that there are at least $3$ vertices in $A_2 \cup A_3$ which are $2$-covered by at least $3$ codewords. Thus, since there is also one such vertex in $A_1$, we have $s_2(\bc) \leq 1+5 \cdot 1/2 + 4 \cdot 1/3 = 29/6$. Assume first that both $(-1,2)$ and $(1,2)$ belong to $C$. If $(-2,2) \in C$, $(0,2) \in C$ or $(2,2) \in C$, there are at least $2$ vertices in $A_1$ with the $I$-set of size at least $3$ and, therefore, we have $s_2(\bc) \leq 19/4$. Hence, according to the rule 7.1, $1/12$ units of share can be shifted from $\bc$ (since $(-1,2)$ and $(1,2)$ are $2$-separated by $C$) and we are done. Thus, without loss of generality, we may assume that $(-1,2) \in C$ and $(1,2) \notin C$. If now $(-3,1) \in C$ and $(-2,2) \in C$, then $|I_2(-1,1)| \geq 4$ and there is no problem since $s_2(\bc) \leq 1+5 \cdot 1/2 + 3 \cdot 1/3 + 1/4 = 19/4$. Furthermore, if the rules 1.1 or 2.1 can be used, then we are again done. Thus, we may assume that $(-2,2) \notin C$ and $(0,2) \notin C$. Therefore, since $(-1,1)$ and $(0,1)$ are $2$-separated by $C$, we obtain that $(-3,1) \in C$. We have now arrive at the constellation illustrated in Figure~\ref{Lemma3UpDownIllustrated}. 
If now one of the vertices $(-4,0)$, $(-3,0)$, $(-3,-1)$, $(-2,-1)$, $(0,-2)$ and $(2,-1)$ belongs to $C$, then by Lemma~\ref{SimpleEstLemma} we immediately have $s_2(\bc) \leq 19/4$. Hence, we may assume that none of these vertices belongs to $C$. Thus, since $(-1,0)$, $(-1,-1)$ and $(1,0)$ are $2$-separated by $C$, we have $(-2,-2) \in C$ and $(3,0) \in C$. If now $(3,-1) \in C$ or $(4,0) \in C$, then again $s_2(\bc) \leq 19/4$ and we are done. Therefore, since $(-1,0)$ and $(1,-1)$ are $2$-separated by $C$, we have $(2,-2) \in C$ and $1/12$ units of share can be shifted from $\bc$ to $(0,-1)$ according to the rule 10. Hence, we have $\ms_2(\bc) \leq 19/4$.

\medskip

3) Finally, assume that $\bc$ is a codeword such that $I_2(\bc) = \{ \bc \}$, and that at least two of the vertices $(-2,1)$, $(2,1)$ and
$(0,-1)$ belong to $C$. Then, by Lemma~\ref{SimpleEstLemma}, we have
$s_2(\bc) \leq 14/3$. This observation completes the proof of the lemma.
\end{proof}

\section*{Acknowledgements}

We would like to thank the anonymous referee for providing constructive comments to improve the paper.

\bibliographystyle{abbrv}

\end{document}